\documentclass{compositio}

\usepackage[utf8]{inputenc}
\usepackage[english]{babel}
\usepackage{newlfont}
\usepackage{amsmath,amssymb,amsfonts,amsbsy,bbm}
\usepackage{mathtools,braket,mathrsfs}
\usepackage[all]{xy}
\usepackage{tikz}
\usetikzlibrary{arrows}

\usepackage{stmaryrd}
\usepackage{hyperref}
\usepackage{upgreek}
\usepackage[new]{old-arrows}
\usepackage{extarrows}
\usepackage{textcomp}

\newcommand{\plectic}[0]{\text{\textmarried}}

\newcommand{\bb}{\mathbb}

\newcommand{\scr}{\mathscr}
\newcommand{\mrm}{\mathrm}
\newcommand*{\bfcdot}{\scalebox{0.6}{$\bullet$}}
\newcommand{\et}{{\acute{\mathrm{e}}\mathrm{t}}}

\newcommand{\Z}{\ensuremath{\mathbb{Z}}}
\newcommand{\OO}{\ensuremath{\mathcal{O}}}
\newcommand{\Q}{\ensuremath{\mathbb{Q}}}

\newcommand{\C}{\ensuremath{\mathbb{C}}}

\newcommand{\p}{\ensuremath{\mathfrak{p}}}
\newcommand{\q}{\ensuremath{\mathfrak{q}}}
\newcommand{\A}{\ensuremath{\mathbb{A}}}

\newcommand{\too}{\longrightarrow}								
\newcommand{\mapstoo}{\longmapsto}

\newcommand{\into}{\hookrightarrow}

\newcommand{\cf}{\mathbbm{1}}
\newcommand{\PP}{\ensuremath{\mathbb{P}^1}}

\newcommand{\G}{\ensuremath{\mathcal{G}}}
\newcommand{\n}{\ensuremath{\mathfrak{n}}}
\newcommand{\m}{\ensuremath{\mathfrak{m}}}

\newcommand{\St}{\textnormal{St}}
\newcommand{\HH}{\textnormal{H}}
\newcommand{\dd}{\textnormal{d}}
\newcommand{\LI}{\mathcal{L}}

\DeclareMathOperator{\Hom}{Hom}

\DeclareMathOperator{\Gal}{Gal}

\DeclareMathOperator{\id}{id}

\DeclareMathOperator{\GL}{GL}
\DeclareMathOperator{\PGL}{PGL}

\DeclareMathOperator{\rec}{rec}
\DeclareMathOperator{\disc}{disc}
\DeclareMathOperator{\Div}{Div}

\DeclareMathOperator{\sm}{sm}
\DeclareMathOperator{\Ev}{Ev}

\DeclareMathOperator{\rk}{rk}

\renewcommand{\det}{\operatorname{det}}
\DeclareMathOperator{\ord}{ord}
\DeclareMathOperator{\pr}{pr}
\DeclareMathOperator{\cont}{ct}

\def\Xint#1{\mathchoice%
{\XXint\displaystyle\textstyle{#1}}%
{\XXint\textstyle\scriptstyle{#1}}%
{\XXint\scriptstyle\scriptscriptstyle{#1}}%
{\XXint\scriptscriptstyle\scriptscriptstyle{#1}}%
\!\int}%
\def\XXint#1#2#3{{\setbox0=\hbox{$#1{#2#3}{\int}$}%
\vcenter{\hbox{$#2#3$}}\kern-.5\wd0}}%

\theoremstyle{plain}
\newtheorem{theorem}{Theorem}[section]
\newtheorem{lemma}[theorem]{Lemma}
\newtheorem{proposition}[theorem]{Proposition}

\newtheorem{conjecture}[theorem]{Conjecture}

\newtheorem{thmx}{Theorem}

\theoremstyle{definition}
\newtheorem{remark}[theorem]{Remark}

\newtheorem{definition}[theorem]{Definition}

\def\Xint#1{\mathchoice
{\XXint\displaystyle\textstyle{#1}}%
{\XXint\textstyle\scriptstyle{#1}}%
{\XXint\scriptstyle\scriptscriptstyle{#1}}%
{\XXint\scriptscriptstyle\scriptscriptstyle{#1}}%
\!\int}
\def\XXint#1#2#3{{\setbox0=\hbox{$#1{#2#3}{\int}$ }
\vcenter{\hbox{$#2#3$ }}\kern-.585\wd0}}
\def\mint{\Xint\times}

\begin{document}

\title{Plectic Stark--Heegner points}

\author{Michele Fornea}
\email{mfornea@math.columbia.edu}
\address{Columbia University, New York, USA.}

\author{Lennart Gehrmann}
\email{lennart.gehrmann@uni-due.de}
\address{Universität Duisburg-Essen, Essen, Germany.}

\classification{11F41, 11F67, 11G05, 11G40.}

\begin{abstract}
	We propose a conjectural construction of global points on modular elliptic curves over arbitrary number fields, generalizing both the $p$-adic construction of Heegner points via $\check{\text{C}}$erednik--Drinfeld uniformization and the definition of classical Stark--Heegner points.  In alignment with Nekov\'a$\check{\text{r}}$ and Scholl's plectic conjectures, we expect the non-triviality of these \emph{plectic Stark--Heegner points} to control the Mordell--Weil group of higher rank elliptic curves. We provide some indirect evidence for our conjectures by showing that higher order derivatives of anticyclotomic $p$-adic $L$-functions compute plectic invariants.
\end{abstract}

\maketitle

\tableofcontents

\section{Introduction}
When $F$ is a totally real number field and $A_{/F}$ is a modular elliptic curve of conductor $\mathfrak{f}_A$, there is a Hilbert cuspform $f_A$ of parallel weight $2$ whose $L$-function coincide with that of the elliptic curve. This equality provides analytic continuation and functional equation for the Hasse--Weil $L$-function of $A_{/F}$. Experience has shown that it is easier to study the Taylor series of $L(A/F,s)$ at the center $s=1$ by additionally considering a quadratic CM extension $E/F$. If the conductor $\mathfrak{f}_A$ is unramified in $E$, then the sign of the functional equation of $L(A/E,s)$ is computed by 
\[
\varepsilon(A/E)=(-1)^{\lvert\Sigma(A/E)\rvert},
\]
where $\Sigma(A/E)$ is a certain set of places of $F$ containing all the Archimedean places of $F$ (given that $E/F$ is a CM extension).
 There are two key formulas describing the first non-trivially-zero term of the Taylor series of $L(A/E,s)$:

\begin{itemize}
\item[\bfcdot] when $\varepsilon(A/E)=+1$, i.e. when the set $\Sigma(A/E)$ has even cardinality, Waldspurger formula expresses the special $L$-value $L(A/E,1)$ in terms of CM points on the \emph{zero dimensional} Shimura variety associated to the totally definite quaternion algebra with ramification set $\Sigma(A/E)$ (\cite{Waldspurger}, \cite{YZZ});

\item[\bfcdot] when $\varepsilon(A/E)=-1$, i.e. when $\Sigma(A/E)$ has odd cardinality, the Gross--Zagier--Zhang formula expresses  the value of the first derivative $L'(A/E,1)$ in terms of CM points points on a \emph{one dimensional} Shimura variety associated to the quaternion algebra of ramification set $\Sigma(A/E)\setminus \{\tau\}$ for some Archimedean place $\tau$ (\cite{GZformula}, \cite{YZZ}).
\end{itemize}
The optimist might guess a pattern emerging from this data and  hope that analogous relations hold in general: is it really possible for CM points on $r$-dimensional quaternionic Hilbert varieties to control the value of $r$-th derivatives  of $L$-functions of elliptic curves? 

\noindent After some reflection, one notes at least two apparent obstacles to this expectation: 
\begin{itemize}
	\item [($i$)] Shimura varieties are defined over reflex fields which might be disjoint from $F$ in general;
	\item [($ii$)] the Chow group of zero cycles $\mrm{CH}_0(X)$ of an $r$-dimensional quaternionic Hilbert variety $X$ should not carry information about modular elliptic curves as soon as $r\ge2$. Indeed, the Beilinson--Bloch--Kato conjecture  (\cite{YLiu}, Section 1.2.1) predicts that the rank of $\mrm{CH}_0(X)$ is essentially controlled by $\mrm{H}_\et^{2r-1}(X_{\overline{\bb{Q}}},\bb{Q}_p(r))$, while the $f_A$-isotypic component of the cohomology of $X$ is concentrated in degree $r$.  
\end{itemize}   
These obstacles may appear to be insurmountable and thus good reasons to reject the optimistic expectation. However, Nekov\'a$\check{\text{r}}$ and Scholl conjecture the existence of a ``plectic'' enhancement of the \'etale cohomology of quaternionic Hilbert varieties that should resolve the issues, and provide generalizations of Siegel units and Heegner points  (\cite{PlecticNS}, \cite{NekRubinfest}). 

\noindent In this article we explore the consequences of the optimistic expectation in a pragmatic way: inspired by Nekov\'a$\check{\text{r}}$ and Scholl's insights we directly propose a conjectural construction of global points on modular elliptic curves via $p$-adic integration.   Ours is a common generalization of both the construction of Heegner points via $\check{\text{C}}$erednik--Drinfeld uniformization (\cite{CDuniformization}) and the cohomological definition of Stark--Heegner points (\cite{IntegrationDarmon}, \cite{Greenberg}, \cite{ArbitraryDarmon}, \cite{AutomorphicDarmon}). In particular, we consider elliptic curves with mulitplicative reduction at a collection of $p$-adic primes.
We expect the non-triviality of these \emph{plectic Stark--Heegner points} to control the Mordell--Weil group of higher rank elliptic curves, and we provide some indirect evidence by showing that higher order derivatives of anticyclotomic $p$-adic $L$-functions compute plectic invariants.
Furthermore, we note that our approach works uniformly over \emph{arbitrary base number fields} $F$  and without a priori restrictions on the type of quadratic extension $E/F$, suggesting that plectic ideas may have wider scope than previously expected.

We refer to \cite{PlecticInvariants} for some \emph{direct numerical evidence}, and a somewhat simplified exposition of the construction of plectic invariants in the case of $E/F$ not a CM extension and both number fields having narrow class number one. Finally, we invite the adventurous reader to have a look at \cite{PlecticJacobians}, where we develop the speculative  framework of \emph{plectic Jacobians}, which recasts the construction of plectic Stark--Heegner points for CM extensions in more geometric terms.
The rest of the introduction is devoted to describing our conjectural construction of global points, and to exploring the relation between plectic invariants and higher derivatives of anticyclotomic $p$-adic $L$-functions.

\subsection{Elliptic curves of higher rank}
Let $F$ be a number field, $A_{/F}$ a modular elliptic curve of conductor $\mathfrak{f}_A$, and $E/F$ a quadratic extension where $\mathfrak{f}_A$ is unramified. A long-standing open problem in the arithmetic of elliptic curves consists in discovering a modular construction of an element $\mrm{P}\in\wedge^rA(E)$ such that $\mrm{P}$ is non-torsion exactly when the algebraic rank $r_\mrm{alg}(A/E)$ equals $r$. During the last decade, Nekov\'a$\check{\text{r}}$ and Scholl have been promoting the idea that a plectic enhancement of the Abel-Jacobi map of quaternionic Hilbert varieties would produce the image of $\mrm{P}\in\wedge^rA(E)$  under the Kummer map
	\[
	\wedge^r\cal{K}\colon \wedge^rA(E)\longrightarrow \otimes^r_{\bb{Q}_p}\mrm{H}^1_f(E,V_p(A)).
	\]
While their plectic conjectures seem to be out of reach at present, the results of this paper may be interpreted as a direct $p$-adic construction of a localization of Nekov\'a$\check{\text{r}}$ and Scholl's cohomology class. To justify our claim, let us fix a rational prime $p$ unramified in $F$ and suppose there is a set  $S=\{\mathfrak{p}_1,\dots,\mathfrak{p}_r\}$ of $r$ distinct $p$-adic $\cal{O}_F$-prime ideals inert in $E$. The local Kummer map induces an isomorphism $\widehat{A}(E_\p)\otimes_{\bb{Z}_p}\bb{Q}_p\overset{\sim}{\too}\mrm{H}^1_f(E_\p,V_p(A))$,
where $\widehat{A}(E_\p)$ denotes the $p$-adic completion of the local points.  Then, by fixing embeddings
$\iota_\mathfrak{p}\colon E\hookrightarrow E_\mathfrak{p}$ for primes $\mathfrak{p}\in S$, we can consider the composition
\[\xymatrix{ \wedge^rA(E)\ar[rr]^-{\wedge^r\cal{K}}\ar@/^2pc/@{.>}[rrr]^-\det&& \otimes^r_{\bb{Q}_p}\mrm{H}^1_f(E,V_p(A))\ar[r]^-{\text{loc}_S} &  \otimes_{\p\in S}\big(\widehat{A}(E_\p)\otimes_{\bb{Z}_p}\bb{Q}_p\big)\\
	}\]
 given by
\[
\det\big(P_1\wedge\dots\wedge P_r\big)=\det \begin{pmatrix}
		\iota_{\mathfrak{p}_1}(P_1)&\dots& \iota_{\mathfrak{p}_r}(P_1)\\
	&\dots&\\
	\iota_{\mathfrak{p}_{1}}(P_r)&\dots& \iota_{\mathfrak{p}_r}(P_r)
	\end{pmatrix}.
\]
The main contribution of this paper is the conjectural construction of elements in the image  $\det(\wedge^rA(E))$ under further assumptions on the conductor of the elliptic curve
presented below.

\begin{remark}\label{extraVanishing}
Unfortunately, the localization process loses information when global classes are annihilated. For instance, if the elliptic curve $A_{/F}$ can be defined over $\bb{Q}$ and $r\ge 2$, the map $\det=\text{loc}_S\circ(\wedge^r\cal{K})$
	\[
	\wedge^rA(\bb{Q})\too \wedge^r\big(\widehat{A}(\bb{Q}_p)\otimes_{\bb{Z}_p}\bb{Q}_p\big)\hookrightarrow\otimes_{\p\in S}\big(\widehat{A}(\bb{Q}_p)\otimes_{\bb{Z}_p}\bb{Q}_p\big)
	\]
	equals zero because $\widehat{A}(\bb{Q}_p)$ has $\bb{Z}_p$-rank one. 
\end{remark}

\noindent  From now on,  suppose that $p_{S}=\prod_{\mathfrak{p}\in S}\mathfrak{p}$ \emph{exactly} divides the conductor $\mathfrak{f}_A$. We can then write
\[
\mathfrak{f}_A=p_{S}\cdot\mathfrak{n}^{\mbox{\tiny $+$}}\cdot\mathfrak{n}^{\mbox{\tiny $-$}},
\]
where $\mathfrak{n}^{\mbox{\tiny $+$}}$ is the product of \emph{all} prime divisors of $\mathfrak{f}_A$ that split in $E$. Let $\infty_1,\dots,\infty_{t}$ denote the real places of $F$ ordered such that the first $n$ are precisely those that split in $E$. Assuming that the ideal $\mathfrak{n}^{\mbox{\tiny $-$}}$ is square-free and denoting its number of prime factors by $\omega(\mathfrak{n}^{\mbox{\tiny $-$}})$, the root number of $A_{/E}$ is conjecturally computed by
\[
\varepsilon(A/E)=(-1)^{r+\omega(\mathfrak{n}^{\mbox{\tiny $-$}})+(t-n)}.
\]
Therefore, the parity conjecture suggests us to impose the following congruence condition
\[
 \omega(\mathfrak{n}^{\mbox{\tiny $-$}})\equiv (t-n) \pmod{2}.
 \]
This last assumption ensures the existence of a quaternion algebra $B/F$ ramified precisely at $\{ \mathfrak{q}\mid \mathfrak{n}^{\mbox{\tiny $-$}}\}\cup\{\infty_{n+1},\dots,\infty_{t}\}$
and admitting an embedding $\psi\colon E\hookrightarrow B$. This data plays a key role in the construction of plectic Stark--Heegner points.
We denote by $T=E^\times/F^\times$ and $G=B^\times/F^\times$ the associated $F$-algebraic groups, and by $\psi\colon T\hookrightarrow G$ the induced embedding. Further, we let $\pi$ be the automorphic representation of $G$ associated to $A_{/F}$ by modularity and the Jacquet--Langlands correspondence.
	The choice of the quaternion algebra $B$ might appear unjustified. However, the relation between plectic points and anticyclotomic $p$-adic $L$-functions -- more precisely the local constraints on linear periods appearing in Waldspurger formula -- suggests that the choice of $B$ made above is the only interesting one for our purposes.

 While developing the necessary machinery, it is natural to consider other invariants associated to the elliptic curve $A_{/F}$ and the quadratic extension $E/F$. These are called \emph{plectic $p$-adic invariants} and have the advantage of being much easier to define while already encoding interesting information.

\subsection{Plectic p-adic invariants}
We set $q$ to be the number of archimedean places at which $B$ is split, i.e.~$q$ is the sum of $n$ and the number of complex places of $F$.
By our assumptions, any arithmetic subgroup of $T(F)$ has rank $q$. The construction of plectic $p$-adic invariants comprises three main steps: 
\begin{itemize}
	\item [$\bfcdot$] the definition of a measure-valued cohomology class $c_{A,\epsilon}^S$ associated to the triple $(A_{/F}, \epsilon, S)$ consisting of the modular elliptic curve $A_{/F}$, a character $\epsilon\colon\pi_0(G_\infty)\to\{\pm1\}$, and $S$;
	\item [$\bfcdot$] the definition of an integrand function arising from the embedding $\psi\colon T\hookrightarrow G$;
	\item [$\bfcdot$] the cap product with a homology class $\vartheta_\chi$ associated to the non-split torus $T$ and a ring class character $\chi$.
\end{itemize}

\subsubsection{Measure-valued cohomology class.}
Let $\p\in S$ and fix an isomorphism $G_\p\cong\PGL_2(F_\p)$. We consider the character $\chi^{\pm}_{\p}\colon G_\p\to \{\pm 1\}$ given by the formula
	\[
	\chi^{\pm}_{\p}(g)=(\pm1)^{\ord_\p(\det(g))}\qquad\forall\ g\in G_\p.
	\]
The $\pm$-Steinberg representation $\St_\p^\pm(\bb{Z})$ of  $G_\p$
 is the space of locally constant $\bb{Z}$-valued functions on $\PP(F_\p)$ modulo constant functions, with action given by 
 \[
 (g\cdot f)(\cdot)=\chi_\p^{\pm}(g)\cdot f(g^{-1}\cdot)\qquad\forall\ g\in G_\p.
 \]
\noindent
We write $S$ as a disjoint union $S=S^+\cup S^-$ where $S^\pm=\{\p\in S\mid \pi_{\p}=\St_\p^\pm(\C)\}$, and put
	\[
	\St_{S}(\bb{Z}):=\bigotimes_{\p\in S^+}\St_\p^+(\bb{Z}) \otimes_\bb{Z} \bigotimes_{\p\in S^-} \St_\p^-(\bb{Z}).
	\]
By a $\bb{Z}$-valued measure on $\bb{P}^1(F_S)$ we mean a finitely additive function defined on compact open subsets of $\prod_{\p\in S}\bb{P}^1(F_\p)$.
Then, the $\bb{Z}$-linear dual $\mrm{Hom}_\bb{Z}(\St_{S}(\bb{Z}),\bb{Z})$ of the $S$-Steinberg representation is naturally identified with a submodule of $\bb{Z}$-valued measures on $\bb{P}^1(F_S)$. 

\noindent To simplify the notation in the introduction, we will ignore all kinds of class number issues.
Theorem \ref{class} shows that there is an $S$-arithmetic subgroup $\Gamma^S\le G(F)$ such that, for every character $\epsilon\colon \pi_0(G_\infty)\to \left\{\pm 1\right\}$, the $[\pi,\epsilon]$-isotypic component of $\mrm{H}^{q}(\Gamma^S,\Hom_\Z(\St_{S}(\bb{Z}),\bb{Z}))$ is finitely generated of rank one.
Therefore, it is natural to define the measure-valued cohomology class $c_{A,\epsilon}^S$ associated to the triple $(A_{/F},\epsilon, S)$ as a generator of its free part, i.e.~
\[
c_{A,\epsilon}^S\in\mrm{H}^{q}\big(\Gamma^S,\Hom_\Z(\St_{S}(\bb{Z}),\bb{Z})\big)_{\pi}^{\epsilon},
\]
which is unique up to sign.

\subsubsection{The integrand function.}
Recall that any $\p\in S$ is inert in the quadratic extension $E/F$.
Hence, the completion of $E$ at $\p$ is the unramified quadratic extension of $F_\p$ and we can consider  
\[
\mathcal{H}_\p(E_\p)=\PP(E_\p)\setminus\PP(F_\p)
\]
onto which the group $G_\p$ acts via M\"obius transformations. The embedding $\psi\colon T\hookrightarrow G$ induces an action of $T_\p$
 on $\mathcal{H}_\p(E_\p)$ with two fixed points $\tau_{\psi,\p}$ and $\bar{\tau}_{\psi,\p}$, which are interchanged by  the Galois group $\Gal(E_\p/F_\p)=\langle\sigma_\p\rangle$. 
 We choose $\tau_{\psi,\p}\in \mathcal{H}_\p(E_\p)$ such that the action of $T_\p$ on the tangent space of $\tau_{\psi,\p}$ is given by the homomorphism
 \[
 T_\p\too E_\p^{\times},\quad t\mapsto t^{1-\sigma_\p}.
 \]
Denote by $\widehat{E}^\times_\p$ the torsion-free part of the $p$-adic completion of $E_\p^\times$ and consider the tensor product $\widehat{E}^\times_{S,\otimes}:=\otimes_{\p\in S}\widehat{E}^\times_\p$ of $\bb{Z}_p$-modules. The $T_S$-equivariant map
\begin{equation}\label{Xi1Intro}
(\Psi^{\diamond}_{S})^{\ast}\colon\Hom_{\Z_p}\big(\St_{S}(\bb{Z}_p),\bb{Z}_p\big)\too \widehat{E}^\times_{S,\otimes}
\end{equation}
is defined using multiplicative integrals
\[
(\Psi^{\diamond}_{S})^{\ast}(\mu):=\mint_{\bb{P}^1(F_S)}\bigotimes_{\p\in S}\bigg(\frac{t_\p-\tau_{\psi,\p}}{t_\p-\bar{\tau}_{\psi,\p}}\bigg) \mrm{d}\mu(t)\ \in  \widehat{E}^\times_{S,\otimes}.
\]
For a subgroup of $T(F)$, the notions of being arithmetic and being $S$-arithmetic agree since $T_S$ is compact.
Thus, the preimage $\Gamma_T=\psi^{-1}(\Gamma^{S})$ under the embedding $\psi\colon T(F)\into G(F)$ is an arithmetic subgroup of $T(F)$.
Restriction followed by the map induced in cohomology by \eqref{Xi1Intro} yields a class
\[
(\Psi^{\diamond}_{S})^{\ast}(c_{A,\epsilon}^S)\in \mrm{H}^q\big(\Gamma_T,\widehat{E}_{S,\otimes}^{\times}\big).
\]
As the abelian group $\Gamma_T$ is a finitely generated of rank $q$, there exists a fundamental class $\vartheta_{\Gamma_T}\in\HH_{q}(\Gamma_T, \Q)$
whose cap product gives an element
\begin{align}\label{capintro}
(\Psi^{\diamond}_{S})^{\ast}(c_{A,\epsilon}^S)\cap \vartheta_{\Gamma_T} \in \widehat{E}_{S,\otimes}^{\times} \otimes_\Z \Q.
\end{align}

\subsubsection{Plectic $p$-adic invariants.}\label{introInv}
 Let $E_\mathfrak{c}/E$ be the narrow ring class field of $E$ of conductor $\mathfrak{c}$ prime to $\frak{f}_A$,
i.e.~the Galois extension of $E$ such that the Artin map induces an isomorphism
\[
\mrm{rec}_E\colon T(F)^{+}\backslash T(\A^{\infty})/U(\mathfrak{c})\xlongrightarrow{\sim}\Gal(E_\mathfrak{c}/E).
\]
Set $\mathcal{G}_{\mathfrak{c}}:=\Gal(E_\mathfrak{c}/E)$ and consider a finite order character $\chi\colon \mathcal{G}_{\mathfrak{c}}\to \overline{\bb{Q}}^{\times}$ of conductor $\mathfrak{c}$ and $\chi_\infty=\epsilon$.
The plectic $p$-adic invariant attached to the pair $(A_{/F},\chi)$ is given by a sum of cap products as in \eqref{capintro} for different embeddings $\psi$ weighted by the character $\chi$.
More precisely, in Section \ref{Normalized} we explain how to construct an adelic and twisted variant $\vartheta_{\chi}$ of the fundamental class, as well as an adelic variant of the map $(\Psi^{\diamond}_{S})^{\ast}$.
Then, the \emph{plectic $p$-adic invariant} for the pair $(A_{/F},\chi)$ is given by the cap product
\[
\mrm{Q}_{A}^{\chi}:=(\Psi^{\diamond}_{S})^{\ast}( c_{A,\epsilon}^S)\cap \vartheta_{\chi}\ \in\  \widehat{E}_{S,\otimes}^{\times}\otimes_\Z\overline{\bb{Q}}.
\]
Now, the elliptic curve $A_{/F}$ has multiplicative reduction at all $\p \in S$.
Therefore, it admits Tate uniformizations $E_\p^{\times}\to A(E_\p)$ and correspondingly maps $\widehat{E}_\p^{\times} \to \widehat{A}(E_\p)$
between the torsion-free parts of the $p$-adic completions.
We denote the tensor product of these maps by
\[
\phi_{\mbox{\tiny $\mrm{Tate}$}} \colon  \widehat{E}_{S,\otimes}^{\times}\too \widehat{A}(E_{S})
\]
where $\widehat{A}(E_{S}):=\bigotimes_{\p\in S}\widehat{A}(E_{\p})$ is a tensor product of $\bb{Z}_p$-modules.
Thus, we can consider the point
\[
\phi_{\mbox{\tiny $\mrm{Tate}$}}\big(\mrm{Q}_{A}^{\chi}\big)\ \in\ \widehat{A}(E_{S})\otimes_\Z \overline{\bb{Q}}.
\]
whose conjectural properties are discussed and supported by numerical experiments in \cite{PlecticInvariants}. 
\begin{remark}
	The elements $\mrm{Q}_{A}^{\chi}$ are not the most refined invariants considered in this paper, but they are already very interesting. They are defined unconditionally for arbitrary base number field $F$ and for any quadratic extension $E$. Moreover, they appear as the value of higher derivatives of anticyclotomic $p$-adic $L$-functions, and Tate's uniformization conjecturally maps them to the $p$-adic image of global points.

\end{remark}

\subsection{Plectic Stark--Heegner points}
We are now ready to address the construction of the more refined  invariants.
In the definition of plectic $p$-adic invariants, we used an integrand function 
\[
\bb{P}^1(F_S)\too\widehat{E}^\times_{S,\otimes},\quad t\mapsto \bigotimes_{\p\in S}\bigg(\frac{t_\p-\tau_{\psi,\p}}{t_\p-\bar{\tau}_{\psi,\p}}\bigg)
\]
that had neither zeros nor poles. Ideally, however, one would want to integrate the function $t\mapsto \otimes_{\p\in S}(t_\p-\tau_{\psi,\p})$ despite it having poles at infinity. It turns out that the latter function naturally lives in a controlled extension $\mathfrak{E}_S$ of the $S$-Steinberg representation (Definition \ref{BreuilExtension}) which allows us to define a $T_S$-equivariant map
\begin{equation}\label{Xi2Intro}
(\Psi_{S})^{\ast}\colon\Hom_{\Z_p}\big(\mathfrak{E}_S,\widehat{A}(E_{S})\big)\to \widehat{A}(E_{S}),\qquad (\Psi_{S})^{\ast}(\mu):=\mu\big(\otimes_{\p\in S}(t_\p-\tau_{\psi,\p})\big).
\end{equation}
For this section and the next, we assume the equality of arithmetic and automorphic $\LI$-invariants of modular elliptic curves over $F$. We note that when $F$ is totally real the equality has recently been established by Rosso and the second named author \cite{GeR}, and independently by Spie\ss\hspace{1mm} \cite{Sp3}.
Under this assumption, the key result is that, at the cost of replacing $\widehat{E}^\times_{S,\otimes}$ with $\widehat{A}(E_{S})$, the measure-valued cohomology class $c_{A,\epsilon}^S$ lifts -- uniquely up to finite torsion -- to a class 
\[
\widehat{c}_{A,\epsilon}^S\in \mrm{H}^q\big(\Gamma^S,\mrm{Hom}_{\bb{Z}_p}(\mathfrak{E}_S,\widehat{A}(E_{S})\big)_{\pi}^\epsilon.
\]

\noindent As before, restriction followed by the map induced in cohomology by \eqref{Xi2Intro} yields a class 
\[
(\Psi_{S})^{\ast}\big(\widehat{c}_{A,\epsilon}^S\big) \in \mrm{H}^q\big(\Gamma_T,\widehat{A}(E_S)\big).
\]
Given a finite order character $\chi\colon \mathcal{G}_{\mathfrak{c}}\to \overline{\bb{Q}}^{\times}$ of conductor $\mathfrak{c}$ with $\chi_\infty=\epsilon$,
the \emph{plectic Stark--Heegner point} attached to the pair $(A_{/F},\chi)$ is given by the cap product 
\[
\mrm{P}_{A}^{\chi}:=(\Psi_S)^{\ast}(\widehat{ c}_{A,\epsilon}^S)\cap \vartheta_{\chi}\ \in\  \widehat{A}(E_{S})\otimes_\Z \overline{\bb{Q}}
\]
of adelic objects as in Section \ref{introInv}.
The point $\mrm{P}_{A}^{\chi}$ does not depend on the choice of the lift $\widehat{c}_{A,\epsilon}^S.$

\noindent Interestingly, we can write down an explicit relation between plectic $p$-adic invariants and plectic Stark--Heegner points.
Recall the ``partial Frobenius'' element $\sigma_\p$, i.e. the non-trivial element of $\Gal(E_\p/F_\p)$.
Since $\sigma_\p$ acts on $\widehat{A}(E_S)$ through the $\p$-th component, the relation
\[
\bigg(\prod_{\p\in S^+}(1-\sigma_\p^{\ast})\prod_{\p\in S^{-}}(1+\sigma_\p^{\ast})\bigg) \mrm{P}_{A}^{\chi}=  \phi_{\mbox{\tiny $\mrm{Tate}$}} \big(\mrm{Q}_{A}^{\chi}\big)
\]
 in $\widehat{A}(E_{S})\otimes_\Z \overline{\bb{Q}}$ is a direct consequence of the definitions.

\subsection{Conjectures}\label{conjectures}
We continue our introduction by stating conjectures that  describe the algebraicity of plectic Stark--Heegner points, a version of Shimura's reciprocity law they are expected to satisfy, and a precise relation to the global arithmetic of elliptic curves. 
Recall we fixed a rational prime $p$, and $S$ a set of $r$ distinct prime ideals of $\cal{O}_F$ above $p$, inert in $E$. We also supposed that 
\[
	\mathfrak{f}_A=p_{S}\cdot\mathfrak{n}^{\mbox{\tiny $+$}}\cdot\mathfrak{n}^{\mbox{\tiny $-$}}\qquad \&\hspace{-0,1em}\qquad  \omega(\mathfrak{n}^{\mbox{\tiny $-$}})\equiv_2 (t-n),
\]
so that the conjecture on the root number of $A_{/E}$ and the parity conjecture imply the congruence $r_\mrm{alg}(A/E)\equiv_2r$. Let  $E_\mathfrak{c}/E$ be the narrow ring class field of $E$ of conductor $\mathfrak{c}$ prime to $\mathfrak{f}_A$.
Under our assumptions, the $\cal{O}_E$-prime ideals generated by the primes in $S$ split completely in $E_\mathfrak{c}^{\mbox{\tiny $+$}}$. Therefore, we can fix an embedding $\iota_\mathfrak{p}\colon E_\mathfrak{c}\hookrightarrow E_\mathfrak{p}$ for each $\mathfrak{p}\in S$ and consider the homomorphism 
\[
\det\colon \wedge^{r} A(E_\mathfrak{c})\longrightarrow  \widehat{A}(E_S)
\]
given by 
\[
\det\big(P_1\wedge\dots\wedge P_{r}\big)=\det \begin{pmatrix}
		\iota_{\mathfrak{p}_1}(P_1)&\dots& \iota_{\mathfrak{p}_r}(P_1)\\
	&\dots&\\
	\iota_{\mathfrak{p}_{1}}(P_r)&\dots& \iota_{\mathfrak{p}_r}(P_r)
	\end{pmatrix}.
\]

\noindent
 Under the assumption $(\mathfrak{c},\mathfrak{f}_A)=1$, for any character $\chi\colon\G_{\mathfrak{c}}\to\overline{\bb{Q}}^\times$ the $L$-functions $L(A/E,\chi,s)$ and  $L(A/E,s)$ share the same root number $\varepsilon(A/E,\chi)=\varepsilon(A/E).$
Thus, if we denote by $r_\mrm{alg}(A/E,\chi)$ the dimension of $A(E_\mathfrak{c})^\chi$, the $\chi$-isotypic component of $A(E_\mathfrak{c})$, we can expect the congruence $r_\mrm{alg}(A/E,\chi)\equiv_2r$ to hold as well.

\begin{conjecture}\label{algebraicity}(Algebraicity $+$ Shimura reciprocity)
If $r_\mrm{alg}(A/E,\chi)\ge r$, then there exists an element $w_{A}^\chi\in \wedge^r A(E_\mathfrak{c})^\chi$ such that 
	\[
	\mrm{P}^{\chi}_{A}=\det\big(w_{A}^\chi\big).
	\]	
\end{conjecture}

\begin{remark}
One could interpret the conjecture as claiming  that the values $L^{(j)}(A/E,\chi,1)$, for $0\le j< r$, measure the obstruction for the algebraicity of the plectic point $\mrm{P}^{\chi}_{A}$. 
\end{remark}
\noindent Conjecture \ref{algebraicity} alone does not pin down precisely the relevance of plectic points for the arithmetic of elliptic curves of higher rank. To clarify it, we state the following Kolyvagin--type conjecture.
 \begin{conjecture}
 Suppose that $r_\mrm{alg}(A/E,\chi)\ge r$, then 
 	\[ 	\mrm{P}^{\chi}_{A}\not=0\quad\implies\quad r_\mrm{alg}(A/E,\chi)=r.
 	\]	
	If the $L$-function $L(A/F,s)$ is primitive, then the converse implication also holds.
 \end{conjecture}
\noindent The assumption on the primitivity of the $L$-function -- i.e. that $L(A/F,s)$ does not factor as a product of automorphic $L$-functions -- is likely too restrictive. Nevertheless, it is concise and takes care of incovenient phenomena (Remark \ref{extraVanishing}).
\begin{remark}
	When $r_\mrm{alg}(A/E,\chi)<r$ we expect the plectic point $\mrm{P}^{\chi}_{A}$ to be non-trivial.
		However, we cannot guess whether the quantity contains any specific arithmetic information.
\end{remark}

\subsection{Derivatives of anticyclotomic p-adic L-functions}
We conclude the introduction by presenting a relation between  plectic $p$-adic invariants and higher derivatives of $p$-adic $L$-functions. Note then that we do not need to assume the equality of arithmetic and automorphic $\cal{L}$-invariants in this section.

\noindent Let $\mathfrak{c}$ be a non-zero ideal of $F$ coprime to $\mathfrak{f}_A$. If we denote by $E_{\mathfrak{c},\mbox{\tiny$S$}}/E$ the union of the narrow ring class fields of $E$ of conductor $\mathfrak{c}\cdot p_S^n$ for $n\geq 0$, then the Artin map induces an isomorphism 
\[
\rec_E\colon \overline{T(F)}^+\backslash T(\A^{\infty})/U(\mathfrak{c})^{S}\xlongrightarrow{\sim}\Gal\big(E_{\mathfrak{c},\mbox{\tiny$S$}}/E\big)
\]
where $\overline{T(F)}^+$ is the closure of $T(F)^+$ in $T(\A^{\infty})/U(\mathfrak{c})^{S}$. We define $\cal{G}_{\mathfrak{c},\mbox{\tiny $S$}}:=\Gal\big(E_{\mathfrak{c},\mbox{\tiny$S$}}/E\big)$.

\noindent The anticyclotomic $p$-adic $L$-function attached to the tuple $(A_{/E},\epsilon,\mathfrak{c},S)$ depends on a choice of base-point $x\in\bb{P}^1(F_S)$ (Remark \ref{uniquenessL-function}) and it is an element of the completed group ring 
\[
\scr{L}_{S}^\epsilon(A/E)_\mathfrak{c} \in \bb{Z}_p\llbracket \cal{G}_{\mathfrak{c},\mbox{\tiny $S$}}\rrbracket.
\]
Such function was constructed in the generality considered here in (\cite{Felix}, \cite{BG2}), building on Waldspurger formula \cite{YZZ} and the work of File-Martin-Pitale \cite{FileMartinPitale} on test-vectors. The anticyclotomic $p$-adic $L$-function $\scr{L}_{S}^\epsilon(A/E)_\mathfrak{c}$ interpolates central $L$-values of twists by finite order characters $\chi\colon\cal{G}_{\mathfrak{c},\mbox{\tiny $S$}}\to\overline{\bb{Q}}^\times$ ramified at every prime in $S$, and it vanishes at all the other characters (Theorem \ref{corolp-adic}).
Moreover, $\scr{L}_{S}^\epsilon(A/E)_\mathfrak{c}$ vanishes to order at least $r=\lvert S\rvert$ at any character of $\mathcal{G}_{\mathfrak{c}}$  (\cite{BG2}, Theorem 5.5), i.e.
\[
\scr{L}_{S}^\epsilon(A/E)_\mathfrak{c}\in I^{r}
\]
where  $I=\ker\big(\Z_p\llbracket\mathcal{G}_{\mathfrak{c},\mbox{\tiny $S$}}^{\mbox{\tiny $+$}}\rrbracket\to \Z_p[\mathcal{G}_{\mathfrak{c}}]\big)$ is the relative augmentation ideal. It is then natural to try to compute the values of the $r$-th derivative at those characters.

\subsubsection{$p$-adic Gross--Zagier formula.}
Inspired by the work of Bertolini and Darmon \cite{CDuniformization}, we prove that the values of the $r$-th derivative of $\scr{L}_{S}^\epsilon(A/E)_\mathfrak{c}$ compute plectic $p$-adic invariants. 
\begin{definition}
	For any character $\chi\colon\G_{\mathfrak{c}}\to\overline{\bb{Q}}^\times$ with $\chi_\infty=\epsilon$ we set
	\[
	\chi\Big(\partial^{r}\scr{L}_{S}^\epsilon(A/E)_\mathfrak{c}\Big):= \scr{L}_{S}^\epsilon(A/E)_\mathfrak{c}\otimes 1\quad\text{in}\quad I^r/I^{r+1}\otimes_\chi\overline{\bb{Q}},
	\]
	where we consider $I^r/I^{r+1}$ as a $\Z[\G_{\mathfrak{c}}]$-module. 
\end{definition}
\noindent Note that the $\chi$-component of the $r$-th derivative is independent of the choice of base-point $x\in\bb{P}^1(F_S)$ made in constructing the $p$-adic $L$-function (Remark \ref{well-definedderivative}).
Now, as for every $\p\in S$ the image of the local Artin map from $T_\p$ to $\G_{\mathfrak{c}}$ is trivial, the function
\[
\dd\rec_\p\colon \widehat{E}_\p^{\times} \to I/I^2,\qquad t\mapsto \rec_\p(t)-1 \pmod{I^2}
\]
is a group homomorphism.
Taking tensor products yields the map
$\dd\rec_S \colon \widehat{E}_{S,\otimes}^{\times} \to I^{r}/I^{r+1}$, which extends to a $\Z[\G_{\mathfrak{c}}]$-linear homomorphism $\dd\rec_S \colon \widehat{E}_{S,\otimes}^{\times}\otimes_{\chi}\overline{\bb{Q}} \to I^{r}/I^{r+1}\otimes_{\chi}\overline{\bb{Q}}$. Then, our $p$-adic Gross--Zagier formula takes the following form.

\begin{thmx}(Theorem \ref{GZtheorem})
For any character $\chi\colon\G_{\mathfrak{c}}\to\overline{\bb{Q}}^\times$ with $\chi_\infty=\epsilon$ the equality
\[
2^{r}\cdot \chi\Big(\partial^{r}\scr{L}_{S}^\epsilon(A/E)_\mathfrak{c}\Big) = \dd\rec_S(\mrm{Q}_{A}^{\chi})
\]
holds in $I^{r}/I^{r+1}\otimes_{\chi}\overline{\bb{Q}}$.
\end{thmx}

\begin{acknowledgements}
We would like to express our deep gratitute to Xavier Guitart and Marc Masdeu for sharing the exploration of the plectic world, and to Matteo Tamiozzo for explaining Nekov\'a$\check{\text{r}}$ and Scholl's conjectures to us. Moreover, we thank Felix Bergunde, Dante Bonolis, Henri Darmon, David Lilienfeldt, Jan Nekov\'a$\check{\text{r}}$, Kartik Prasanna, Tony Scholl, Nicolas Simard, Ari Shnidman, Christopher Skinner and Jan Vonk for their support and for many enriching conversations. While working on this article the first named author was a Simons Junior Fellow, and the second named author was visiting McGill University, supported by Deutsche Forschungsgemeinschaft.
\end{acknowledgements}

\section{Notations and setup}
All rings are commutative and unital.
The group of invertible elements of a ring $R$ will be denoted by $R^{\times}$.
For a group $H$, we write $R[H]$ for the group algebra of $H$ with coefficients in $R$.
If $H=\varprojlim_n H_n$ is a pro-finite group, we let
\[\Z_p\llbracket H\rrbracket=\varprojlim_n \Z_p[H_n]\]
be the completed group algebra of $H$ with coefficients in $\Z_p$.
Given a group homomorphism $\chi\colon H\to R^{\times}$, we let $R(\chi)$ denote the representation of $H$ whose underlying $R$-module is $R$ itself and on which $H$ acts via the character $\chi$.
If $N$ is another $R[H]$-module, we put $N(\chi)=N\otimes_{R}R(\chi)$.

\noindent Given topological spaces $X,Y$, we will write $C(X,Y)$ for the set of continuous functions from $X$ to $Y$, and $C_{c}(X,Y) \subseteq C(X,Y)$ for the subset of continuous functions with compact support.
For a set $X$ and a subset $U\subseteq X$ the characteristic function $\cf_{U}\colon X\to \left\{0,1\right\}$ is defined by
\begin{align*}
 \cf_{U}(x) = \begin{cases}
	       1 & \mbox{if } x\in U,\\
	       0 & \mbox{else}.
	      \end{cases}
\end{align*}

\noindent\textbf{A number field.}
Throughout the article we fix a number field $F$ with ring of integers $\OO_{F}$.
We denote the set of archimedean places of $F$ by $\Sigma_{\infty}$.
For a non-zero ideal $\mathfrak{a} \subseteq \OO_F$ we set $N(\mathfrak{a})=\left|\OO_F/\mathfrak{a}\right|$.
If $v$ is a place of $F$, we denote by $F_{v}$ the completion of $F$ at $v$.
If $\p$ is a finite place, we let $\OO_\p$ denote the valuation ring of $F_\p$ and write $\ord_{\p}$ for the normalized valuation.
For a finite set $\Sigma$ of places of $F$ we define the "$\Sigma$-truncated adeles" $\A^{\Sigma}$ as the restricted product of the completions $F_{v}$ over all places $v$ which are not in $\Sigma$. We often write $\A^{\Sigma,\infty}$ instead of $\A^{\Sigma\cup \Sigma_{\infty}}$.
If $H$ is an algebraic group over $F$ and $v$ is a place of $F$, we write $H_v=H(F_v)$ and put $H_\Sigma=\prod_{v\in \Sigma} H_v$.
Similarly as before, we abbreviate $H_\infty=H_{\Sigma_\infty}$.

\vspace{3mm}
\noindent
\textbf{A set and an elliptic curve}.
We fix a rational prime $p$ and a set of distinct primes 
\begin{equation}
	S=\{\mathfrak{p}_1,\dots,\mathfrak{p}_r\}
\end{equation}
 of $F$ lying over $p$. Moreover, we consider an elliptic curve $A$ over $F$ of conductor $\mathfrak{f}_A\subseteq \OO_{F}$ such that $\p$ divides $\mathfrak{f}_A$ \emph{exactly} once for every $\p \in S$ or, equivalently, that $A_{/F}$ has multiplicative reduction for all primes $\p\in S$.
We assume that $A$ is modular, meaning that there exists a cuspidal automorphic representations $\Pi_A$ of $\PGL_2(\A)$ attached to $A$, which is cohomological with respect to the trivial coefficient system and has conductor $\mathfrak{f}_A$.
In particular, for every $\p \in S$ the local component $\Pi_{A,\p}$ is an unramified twist of the Steinberg representation.

\vspace{3mm}
\noindent
\textbf{A quadratic extension}.
We fix a quadratic extension $E/F$ such that 
\begin{itemize}
	\item [$\bfcdot$] every prime dividing the conductor $\mathfrak{f}_A$ is unramified in $E$, and 
	\item [$\bfcdot$] the primes in $S$ are inert.
\end{itemize} 
For $\p\in S$ the non-trivial element of $\Gal(E_\p/F_\p)$ will be denoted by $\sigma_\p$.
We may decompose
\begin{equation*}
\mathfrak{f}_A=p_S\cdot\mathfrak{n}^{\mbox{\tiny $+$}}\cdot\mathfrak{n}^{\mbox{\tiny $-$}},
\end{equation*}
where $\mathfrak{n}^{\mbox{\tiny $+$}}$ and $\mathfrak{n}^{\mbox{\tiny $-$}}$ are coprime, and such that a prime divisor of $\mathfrak{f}_A$ divides $\mathfrak{n}^{\mbox{\tiny $+$}}$ if and only if it split in $E$.
Let $\infty_1,\ldots, \infty_t$ denote the real places of F ordered such that the first $n$ are precisely those
that split in $E$.
We suppose that the ideal $\mathfrak{n}^{\mbox{\tiny $-$}}$ is square-free and denote by $\omega(\mathfrak{n}^{\mbox{\tiny $-$}})$ the number of its prime factors, then the set
\[
\Sigma(A/E)=S\cup\big\{\mathfrak{q}\mid \mathfrak{n}^{\mbox{\tiny $-$}}\big\}\cup\big\{\infty_{n+1},\dots,\infty_t\big\}
\]
controls the root number of $A_{/E}$ which is conjecturally equal to
\[
\varepsilon(A/E)=(-1)^{r+\omega(\mathfrak{n}^{\mbox{\tiny $-$}})+(t-n)}.
\]
We assume that $\omega(\mathfrak{n}^{\mbox{\tiny $-$}}) \equiv (t-n) \pmod{2}$.

\vspace{3mm}
\noindent
\textbf{Algebraic groups}.
Under our assumptions, there exists a quaternion algebra $B_{/F}$ with
\begin{itemize}
	\item [$\bfcdot$] $\disc(B)=\mathfrak{n}^{\mbox{\tiny $-$}}$, and
	\item [$\bfcdot$]  $B$ is ramified at infinity only at the real places $\infty_{n+1},\ldots,\infty_{t}.$
\end{itemize} 
Note that $B$ is split at all primes in $S$.
We denote by $G=B^\times/F^\times$ the associated $F$-algebraic group of units modulo center.
The automorphic representation $\Pi_A$ admits a Jacquet-Langlands transfer $\pi$ to $G$.
The local component $\pi_\q$ of $\pi$ is one-dimensional for every prime $\q\mid\disc(B)$. 
Furthermore, there exists an embedding of $E$ into $B$ that induces an embedding $\psi$ of the $F$-algebraic torus $T=E^\times/F^\times$ into $G$.
Let $q$ be the number of archimedean places of $F$ at which $B$ is split.
By Dirichlet's unit theorem, the rank of every arithmetic subgroup of $T(F)$ is equal to $q$.
For every prime $\q\nmid \disc(B)$ we fix an isomorphism of $B_\q$ with the algebra $\mrm{M}_2(F_\q)$ of two-by-two matrices over $F_\q$ obtaining isomorphisms 
\begin{equation}\label{choseniso}
G_\q\cong \PGL_2(F_\q)\qquad \forall\ \q\nmid \disc(B).
\end{equation}
We write $G_{\infty}^{+}\subseteq G_{\infty}$ for the connected component of the identity, i.e.~the subgroup of elements whose reduced norm is positive for every real place of $F$.
If $H$ is any subgroup of $G_\infty$ we set 
\[
H^{+}:=H\cap G_{\infty}^{+}.
\]

\vspace{3mm}
\noindent
\textbf{Congruence subgroups}.
Let $\n\subseteq\OO_F$ be a non-zero ideal coprime to $\disc(B)$. For every prime $\q$ of $F$ not dividing the discriminant of $B$ we define the compact open subgroup $K_\q(\n)\subseteq G_\q$ as the image  under (\ref{choseniso}) of the subgroup
\[
\left\{g \in \PGL_2(\OO_\q)\mid g \mbox{ upper triangular} \bmod \n\right\}.
\]
For $\q\mid \disc(B)$ we let $K_\q(\n)$ be the image in $G_\p$ of the group of units of the maximal order in $B_\q$.
If $\Sigma$ is a finite set of primes of $F$ we define
\[
K(\n)^{\Sigma}:=\prod_{\q\notin \Sigma}K_\q(\n)\subseteq G(\A^{\Sigma,\infty}).
\]
If $\Sigma$ is the empty set, we drop it from the notation. Further, we set 
\[
	\mathfrak{f}=\mathfrak{f}_A/\disc(B).
\]
Then, if $\pi^{\infty}$ denotes the finite part of $\pi$, Casselman's result on newforms gives (\cite{Ca}) 
\[
	\dim_\C(\pi^\infty)^{K(\mathfrak{f})}=1.
\]
For a finite place $\q$ of $F$ all $\OO_{\q}$-orders in $E_\q$ are of the form $\OO_{\q} + \q^m \OO_{E,\q}$ for some $m\geq 0$ and $\OO_{E,\q}$ the maximal $\OO_{\q}$-order in $E_\q$.
We write $U_{\q}^{(m)}$ for the image of $(\OO_{\q} + \q^m \OO_{E,\q})^{\times}$ in $T_{\q}$, and given a non-zero ideal $\m \subseteq \OO_F$, a finite set $\Sigma$ of places of $F$, we define
\[
U(\m)^{\Sigma} := \prod_{\q \notin \Sigma} U_{\q}^{(\ord_\q(m))}.
\]
The preimage of $K(\mathfrak{f})^{S}$ under $\psi$ is of the form $U(\mathfrak{c})^{S}$ for a non-zero ideal $\mathfrak{c}\subseteq \OO_F$ that is coprime to $\disc(B)$ and all primes in $S$. In addition, we assume that $\mathfrak{c}$ is even coprime to $\mathfrak{f}_A=\disc(B)\mathfrak{f}$.


\section{A cohomology class valued in harmonic cochains}

\subsection{Cohomology of arithmetic groups}\label{Arithmetic}
Let $\n\subseteq \mathcal{O}_F$ be a non-zero ideal that is coprime to $\disc(B)$ and $N$ an abelian group.
The $\Z$-module of (continuous) functions $C\big(G(\A^\infty)/K(\n),N\big)$
 carries a natural $G(F)$-action given by
\[
(\gamma\cdot\Phi)(g)=\Phi(\gamma^{-1}g).
\]
As the double quotient $G(F)^{+}\backslash G(\A^\infty)/K(\n)$ is finite by Borel's theorem (\cite{BoCN}, Theorem 5.1),
we can fix a system of representatives $h_1,\ldots,h_\ell$ and put $\Gamma_{i}=G(F)^{+}\cap h_i K(\n) h_i^{-1}.$
Then, Shapiro's lemma yields natural isomorphisms
\begin{align}\label{Shapiro}
\HH^{d}\big(G(F)^{+},C\big(G(\A^\infty)/K(\n),N\big)\big)\xlongrightarrow{\cong}\bigoplus_{i=1}^{\ell}\HH^{d}(\Gamma_{i}, N)
\end{align}
for every $d\geq 0.$
Thus, if $N$ is a $\Q$-vector space or if the arithmetic groups $\Gamma_{i}$ are torsion-free, the cohomology groups $\HH^{d}\big(G(F)^{+},C\big(G(\A^\infty)/K(\n),N\big)\big)$ are naturally isomorphic to the singular cohomology with coefficients in $N$ of the locally symmetric space of level $K(\n)$ associated to $G$. We believe it justifies the following notation
\begin{equation}\label{notation}
\mrm{H}^d\big(X_G(\n),N\big):= \HH^{d}\big(G(F)^{+},C\big(G(\A^\infty)/K(\n),N\big)\big).
\end{equation}
Since arithmetic groups are of type (\hspace{-0.6mm}\emph{VFL}) (\cite{BS}, Section 11.1) one deduces the following two statements (\cite{Se2}, $p.101$):
\begin{itemize}
	\item [$\bfcdot$] for $N$ a finitely generated module over a Noetherian ring $R$, the $R$-module 
	\[
	\mrm{H}^d\big(X_G(\n),N\big)
	\] is finitely generated;
	\item [$\bfcdot$] if $N$ is a flat $\Z$-module, then the canonical map
\begin{equation}\label{flat}
\mrm{H}^d\big(X_G(\n),\Z\big)\otimes_{\Z}N \xlongrightarrow{\cong} \mrm{H}^d\big(X_G(\n),N\big)
\end{equation}
is an isomorphism. 
\end{itemize}
\noindent
These cohomology groups carry a natural action of the finite, abelian, $2$-torsion group 
\[
G(F)/G(F)^+=G_\infty/G_\infty^{+}=\pi_0(G_\infty).
\]
Given a character $\epsilon\colon \pi_0(G_\infty)\to \{\pm 1\}$ and a $\pi_0(G_\infty)$-module $M$ we denote by $M^{\epsilon}$ the $\epsilon$-isotypic component of $M$.
Let $\Sigma$ be a finite set of primes containing all primes dividing  $\mrm{disc}(B)$.
The cohomology groups $\mrm{H}^d\big(X_G(\mathfrak{n}),N\big)$ carry an action of the Hecke algebra
\[
\mathbb{T}^{\Sigma}:=C_c\big(K(\mathfrak{n})^{\Sigma}\backslash G(\A^{\Sigma,\infty})/K(\n)^\Sigma,\Z\big)
\]
which commutes with the action of $\pi_0(G_\infty)$.
Let $\lambda_\pi\colon \mathbb{T}^{\Sigma} \to \Z$ be the character of  associated to $\pi$ by the following rule
\[
t.v=\lambda_\pi(t)\cdot v\qquad \forall\ t\in \mathbb{T}^{\Sigma},\quad\forall\ v\in(\pi^{\infty})^{K(\mathfrak{f})}.
\]
Given a $\mathbb{T}^{\Sigma}$-module $M$ we put
\[
M_\pi=\left\{m\in M \mid t.m=\lambda_A(t)\cdot m\quad \forall t\in \mathbb{T}^{\Sigma} \right\}.
\]

\begin{proposition}\label{arithclass}
For any flat $\Z$-algebra $R$ and any character  $\epsilon\colon \pi_0(G_\infty)\to \left\{\pm 1\right\}$ we have
\[
\rk_R\hspace{0.2mm} \mrm{H}^d\big(X_G(\mathfrak{f}),R\big)_{\pi}^{\epsilon}=\begin{cases}
1 & \mbox{for } d=q\\
0 & \mbox{for } d<q.
\end{cases}
\]
Furthermore, for every proper divisor $\n$ of $\mathfrak{f}$ and every $d\ge0$ we have
\[
\rk_R\hspace{0.2mm} \mrm{H}^d\big(X_G(\mathfrak{n}),R\big)_{\pi}^{\epsilon}=0.
\]
\end{proposition}
\begin{proof}
The groups $\mrm{H}^d\big(X_G(\mathfrak{f}),\Z\big)_{\pi}^{\epsilon}$ are finitely generated, and by the flat base-change \eqref{flat}, to prove the first claim it suffices to show that 
\[
\dim_\C\hspace{0.2mm}\mrm{H}^d\big(X_G(\mathfrak{f}),\C\big)_{\pi}^{\epsilon}=\begin{cases}
1 & \mbox{for } d=q\\
0 & \mbox{for } d<q
\end{cases}.
\]
This last statement is a direct consequence of Matsushima's formula, strong multiplicity one and standard calculations in $(\mathfrak{g},K_\infty)$-cohomology.
The proof of the second claim is similar.
\end{proof}

\begin{remark}
Let $\mathbb{T}^\Sigma_\Q=\mathbb{T}^\Sigma\otimes\Q$ be the rational Hecke algebra.
The kernel $\m_{\pi}$ of the homomorphism $\lambda_{\pi}\colon \mathbb{T}_\Q^\Sigma \to \Q$ is a maximal ideal.
Since $\pi$ is cuspidal one can prove that
\[
\mrm{H}^d\big(X_G(\mathfrak{n}),\Q\big)_{\pi}=\mrm{H}^d\big(X_G(\mathfrak{n}),\Q\big)_{\m_{\pi}},
\]
where the subscript $\m_{\pi}$ means localization.
We implicitly use this equality throughout the article.
\end{remark}

\subsection{The Steinberg representation and harmonic cochains}\label{Steinberg}
We recall the relation between the Steinberg representation and harmonic cochains on the Bruhat--Tits tree $\scr{T}_\p$ of $G_\p$.
Let $\p$ be a prime in $S$ and $M$ an abelian group.
We remind ourselves that the $\pm$-Steinberg representation $\St_\p^\pm(M)$ of $G_\p$
 is the space of locally constant $M$-valued functions on $\PP(F_\p)$ modulo constant functions, with action given by 
 \[
 (g.f)(x)=\chi_\p^{\pm}(g)\cdot f(g^{-1}x)\qquad\forall\ g\in G_\p,\ x\in \PP(F_\p).
 \]
where $\chi^{\pm}_{\p}\colon G_\p\to \{\pm 1\}$ is the character given by
	\[
	\chi^{\pm}_{\p}(g)=( \pm 1)^{\ord_\p(\det(g))}.
	\]
The Bruhat--Tits tree $\scr{T}_\p$ is a homogeneous oriented tree whose set of vertices $\scr{V}_\p$ consists of homothety classes of $\OO_\p$-lattices in $F_\p\oplus F_\p$.
Two vertices $v_1, v_2\in \scr{V}_\p$ are connected by an oriented edge $e\in\scr{E}_\p$ with source $s(e)=v_1$ and target $t(e)=v_2$ if there are representatives $\Lambda_i$  of $v_i$ for $i=1,2$ such that 
\[
\p\Lambda_2\subsetneq\Lambda_1\subsetneq\Lambda_2.
\]
Given and oriented edge $e\in\scr{E}_\p$ we denote by $\overline{e}$ be the same edge as $e$ but with opposite orientation.
The natural action of $G_\p$ on oriented edges $\scr{E}_\p$ (resp. on the vertices $\scr{V}_\p$) is transitive. We denote by $K_\p^1$ (resp. $K_\p^0$) the stabilizer of the edge $e_\circ$, from the vertex $v_\circ$ corresponding to $\OO_\p\oplus\OO_\p$ to the the vertex $\widehat{v}_\circ$ corresponding to $\mathfrak{p}\OO_\p\oplus\OO_\p$ (resp. of the vertex $v_\circ$).
Thus we have natural identifications
\[
\scr{V}_\p\cong G_\p/K_\p^0\quad\text{and}\quad \scr{E}_\p\cong G_\p/K_\p^1.
\]
Consider the natural $\GL_2(F_\p)$-equivariant projection
\[
\pr\colon F_\p^2\setminus\{0\}\too \PP(F_\p),\quad (x,y)\mapstoo [x:y].
\]
Given an edge $e=(s,t)\in\scr{E}_\p$, choose lattices $\Lambda_s$, $\Lambda_t$ of $F_\p\oplus F_\p$ representing $s$,$t$, such that 
\[
\p\Lambda_{t}\subsetneq\Lambda_{s}\subsetneq\Lambda_{t}.
\]
If we set   $\Lambda'=\Lambda\setminus\p\Lambda$ for any lattice $\Lambda$ of $F_\p\oplus F_\p$,
then $U_e=\pr(\Lambda_s'\cap \Lambda_t')$
is a compact open subset of $\PP(F_\p)$ depending only on the edge $e\in\scr{E}_\p$.
It is easy to see that
\begin{equation}\label{complement}
U_{\overline{e}}=\PP(F_\p)\setminus U_e.
\end{equation}
Moreover, the collection $\left\{U_e\right\}_{e\in\scr{E}_\p}$ forms a basis of the $p$-adic topology of $\PP(F_\p)$. 

\begin{remark}
For any $\gamma\in G_\p$ and any $e=(s,t)\in\scr{E}_\p$ we have
	\[
	U_{\gamma.e}=\pr\left(\gamma(\Lambda_s'\cap\gamma\Lambda_t'\right))=\gamma(\pr\left(\Lambda_s'\cap\Lambda_t'\right))=\gamma(U_e).
	\]
\end{remark}

\subsubsection{Harmonic cochains.}
For any edge $e\in\scr{E}_\p$ we denote by $\delta_e\in C_c(\scr{E}_\p,\Z)$ the function 
\[
\delta_e(e')= \begin{cases}
	1&\text{if}\ e'=e\\
	0&\text{if}\ e'\not= e.
\end{cases}
\]
We consider the surjective $G_\p$-equivariant map
\[
\Ev_\p^\pm\colon C_c(\scr{E}_\p,\Z)\to \St_\p^\pm,\quad \delta_e\mapsto  \chi_\p^{\mbox{\tiny $\pm$}}(g_e)\cdot\cf_{U_e}
\]
where $g_e\in G_\p$ satisfies $g_e.e_\circ=e$.
The map $e\mapsto \overline{e}$ induces an involution $W_\p$ on $C_c(\scr{E}_\p,\Z)$, and by \eqref{complement} we have $\cf_{U_e}=-\cf_{U_{\overline{e}}}$ in $\St_\p^\pm$.
Hence, the map $\Ev_\p^\pm$ factors through the quotient
\[
C_c(\scr{E}_\p,\Z)_{\pm}:=C_c(\scr{E}_\p,\Z)/(1\pm W_\p)C_c(\scr{E}_\p,\Z).
\]
Furthermore, the $G_\p$-equivariant map
\[
C_c(\scr{V}_\p,\Z)\to C_c(\scr{E}_\p,\Z),\quad  f\mapsto [e\mapsto f(s(e))]
\]
fits in a short exact sequence (\cite{Sp}, Section 3.4) 
\begin{equation}\label{resolution}
0\too C_c(\scr{V}_\p,\Z) \too C_c(\scr{E}_\p,\Z)_{\pm}\xlongrightarrow{\Ev_\p^\pm} \St_\p^\pm\too 0.
\end{equation}

\noindent
 Recall that for every discrete set $\Delta$ and any abelian group $N$  there is a canonical isomorphism
\begin{equation}\label{continduality}
C(\Delta,N)\xlongrightarrow{\cong}\Hom_\Z(C_c(\Delta,\Z),N),\quad f\mapsto[g\mapsto \sum_{d\in\Delta}f(d)g(d)].
\end{equation}
Therefore, the exact sequence \eqref{resolution} induces a dual short exact sequence of $G_\p$-modules
\begin{equation}\label{dualresolution}
0\too \Hom_\Z\big(\St_\p^\pm,\bb{Z}\big)\xlongrightarrow{(\Ev_\p^\pm)^{\ast}} C(\scr{E}_\p,\bb{Z})^{W_\p=\mp1}\too C(\scr{V}_\p,\bb{Z})\too 0,
\end{equation}
where the last map is given by $f\mapsto\left[v\mapsto \sum_{s(v)=e}f(e)\right]$, giving  the standard identification of the dual of the Steinberg representation with harmonic cochains on the tree.

\subsection{Cohomology of S-arithmetic groups}\label{sarith}
Let $\Sigma$ a finite set of primes of $F$ and $\n\subseteq \mathcal{O}_F$ a non-zero ideal, both coprime to $\disc(B)$.
Further, fix a ring $R$, an $R$-module $N$ and an $R[G_\Sigma]$-module $M$.
There is a natural $G(F)$-action on the $R$-module $C(G(\A^{\Sigma,\infty})/K(\n)^\Sigma,\Hom_R(M,N))$ given by
\[
(\gamma\cdot \Phi)(g)(m)=\Phi(\gamma^{-1}g)(\gamma^{-1}m),
\]
where $G(F)$ acts on $M$ via the diagonal embedding $G(F)\into G_\Sigma$.
In accordance with \eqref{notation} we set
\[
\HH_R^d(X_G^\Sigma(\n),M,N):=\HH^{d}(G(F)^{+},C(G(\A^{\Sigma,\infty})/K(\n)^\Sigma,\Hom_R(M,N))).
\]
Note that for all $\q\in\Sigma$ we have
\[
\HH_R^d(X_G^\Sigma(\n),M,N)=\HH_R^d(X_G^\Sigma(\q\n),M,N)
\]
As in \eqref{Shapiro}, by choosing representatives  $h_1,\ldots,h_{\ell_\Sigma}$ of the double quotient $G(F)^{+}\backslash G(\A^{\Sigma,\infty})/K(\n)^\Sigma$ and putting $\Gamma_{i}^\Sigma=G(F)^{+}\cap h_i K(\n)^\Sigma h_i^{-1}$, Shapiro's lemma yields natural isomorphisms 
\begin{equation}
\HH_R^d\big(X_G^\Sigma(\n),M,N\big)\xlongrightarrow{\cong}\bigoplus_{i=1}^{\ell_\Sigma}\HH^{d}\big(\Gamma^{\Sigma}_i,\Hom_R(M,N)\big).
\end{equation}

\begin{remark}\label{changelevel}
Let $\q$ be a prime in $\Sigma$.
Suppose that the $R[G_\Sigma]$-module $M$ is of the form
\[
M=C_c(G_\q/K_\q(\n),R)\otimes_R M^{\q}
\]
for some $R[G_{\Sigma\setminus\{\q\}}]$-module $M^{\q}$.
Then, using \eqref{continduality}, one deduces a canonical isomorphism
\[
C(G(\A^{\Sigma,\infty})/K(\n)^\Sigma,\Hom_R(M,N))\cong C(G(\A^{\Sigma\setminus\{\p\},\infty})/K(\n)^{\Sigma\setminus\{\p\}},\Hom_R(M,N)),
\]
which in turns induces isomorphisms for every $d\ge0$
\[
\HH_R^d(X_G^\Sigma(\n),M,N)\cong\HH_R^d(X_G^{\Sigma\setminus\{\q\}}(\n),M^{\q},N).
\]
\end{remark}

\noindent Let $\chi\colon G_\Sigma\to \{\pm 1\}$ be a continuous character.
Since $\Sigma$-arithmetic groups are of type (\hspace{-0.6mm}\emph{VFL}) (cf.~\cite{BS2}) one deduces the following two statements as in the arithmetic case:
\begin{itemize}
	\item [$\bfcdot$] if $N$ is a finitely generated module over a Noetherian ring $R$, the $R$-module 
	\[
	\HH_R^d(X_G^\Sigma(\n),R(\chi),N)
	\] is finitely generated;
	\item [$\bfcdot$] if $N$ is a flat $R$-module then the canonical map
\[
\HH_R^d(X_G^\Sigma(\n),R(\chi),R)\otimes_{R}N \xlongrightarrow{\cong} \HH_R^d(X_G^\Sigma(\n),R(\chi),N)
\]
is an isomorphism.
\end{itemize}

\begin{definition}
Any subset $\Sigma\subseteq S$ can be written as a disjoint union $\Sigma=\Sigma^+\cup\Sigma^-$ where
	\[
	\Sigma^\pm=\big\{\p\in\Sigma\mid \pi_{A,\p}=\St_\p^\pm(\C)\big\}.
	\]
	Then, for any abelian group $M$ we define
	\[
	\St_{\Sigma}(M):=\bigotimes_{\p\in\Sigma^+}\St_\p^+(M) \otimes_\bb{Z} \bigotimes_{\p\in\Sigma^-} \St_\p^-(M).
	\]
	When $M=\bb{Z}$ we simply write $\St_\Sigma$ for $\St_\Sigma(\Z)$. Further, we put
	\[
	\chi_\Sigma:=\prod_{\p\in\Sigma^+}\chi_\p^+\times\prod_{\p\in\Sigma^-}\chi_\p^-.
	\]
\end{definition}

\noindent By induction on the size of $\Sigma'\subseteq\Sigma\subseteq S$, using the long exact sequence in cohomology coming from the short exact sequence \eqref{dualresolution}, and the general results about the cohomology of $\Sigma$-arithmetic groups mentioned above, we deduce that:
\begin{itemize}
	\item [$\bfcdot$] the cohomology groups $\mrm{H}_\Z^d\big(X_G^\Sigma(\mathfrak{n}),\St_{\Sigma'}(\chi_{\Sigma\setminus\Sigma'}),\Z\big)$ are finitely generated $\Z$-modules;
	\item [$\bfcdot$] the
canonical homomorphism
\begin{equation}\label{flat2}
\mrm{H}^d_\Z\big(X_G^\Sigma(\mathfrak{n}),\St_{\Sigma'}(\chi_{\Sigma\setminus\Sigma'},\Z\big)\otimes_{\bb{Z}} N\xlongrightarrow{\cong} \mrm{H}^d_R\big(X_G^\Sigma(\mathfrak{n}),\St_{\Sigma'}(R)(\chi_{\Sigma\setminus\Sigma'}),N\big)
\end{equation}
is an isomorphism for every flat $\Z$-algebra $R$ and any flat $R$-module $N$. 
\end{itemize}

\begin{proposition}\label{class}
Let $R$ be a Noetherian $\Z$-flat algebra.
Let $\Sigma'\subseteq \Sigma$ be subsets of $S$ and $\epsilon\colon \pi_0(G_\infty)\to \{\pm 1\}$ a character, then 
\[
\rk_R\HH^{d}_R\left(X_G^\Sigma(\mathfrak{f}),\St_{\Sigma'}(R)(\chi_{\Sigma\setminus\Sigma'}),R\right)_{\pi}^{\epsilon}=\begin{cases}
1 & \mbox{for } d=q+\lvert\Sigma\setminus\Sigma'\rvert\\
0 & \mbox{for } d<q+\lvert\Sigma\setminus\Sigma'\rvert.
\end{cases}
\]
\end{proposition}
\begin{proof}
We begin with the case $\Sigma'=\Sigma$ which is a variant of (\cite{Sp}, Proposition 5.8).
For the convenience of the reader we include a sketch of the proof:
 by flat base-change it suffices to prove
\[
\dim_\Q\HH^{d}_\Q\big(X_G^\Sigma(\mathfrak{f}),\St_{\Sigma}(\Q),\Q\big)_{\pi}^{\epsilon}=\begin{cases}
1 & \mbox{for } d=q\\
0 & \mbox{for } d<q.
\end{cases}
\]
The proof is by induction on the size of $\Sigma$, the base case $\Sigma=\emptyset$ being the first claim of Proposition \ref{arithclass}.
Suppose $\Sigma\not=\emptyset$ and choose $\p\in \Sigma$.
The short exact sequence \eqref{resolution} for $\p$ induces the following long exact sequence:
\begin{align*}
\cdots
\too &\HH^{d}_\Q\big(X_G^\Sigma(\mathfrak{f}),\St_{\Sigma}(\Q),\Q\big)
\too \HH^{d}_\Q\big(X^{\Sigma\setminus\{\p\}}(\mathfrak{f}),\St_{\Sigma\setminus\{\p\}}(\Q),\Q\big)^{W_\p=\mp1}\\
\too &\HH^{d}_\Q\big(X^{\Sigma\setminus\{\p\}}(\p^{-1}\mathfrak{f}),\St_{\Sigma\setminus\{\p\}}(\Q),\Q\big)
\too \HH^{d+1}_\Q\big(X_G^\Sigma(\mathfrak{f}),\St_{\Sigma},\Q\big)
\too\cdots
\end{align*}
Thus, it is enough to prove that for all $d\geq 0$
\begin{align}\label{toprove}
\HH^{d}_\Q\big(X_G^{\Sigma\setminus\{\p\}}(\p^{-1}\mathfrak{f}),\St_{\Sigma\setminus\{\p\}}(\Q),\Q\big)_{\pi}^\epsilon=0.
\end{align}
Using the long exact sequence induced by \eqref{resolution} for the remaining $\q\in \Sigma$ we may compute the left hand side of \eqref{toprove} in terms of cohomology groups of the form $\HH^{\ast}(X_G(\n),\Q)_{\pi}^\epsilon$ with levels $\n$ that are strict divisors of $\mathfrak{f}.$
But these groups vanish by the second claim of Proposition \ref{arithclass}.

\noindent The general case $\Sigma'\subset\Sigma$ is a variant of (\cite{Sp}, Lemma 6.2 (b)).
Again we just give a sketch of the proof:
it is by induction on $|\Sigma\setminus \Sigma'|$, the case $|\Sigma\setminus \Sigma'|=0$ being proven above.
Assume the claim for some $\Sigma'$.
For $\p\in\Sigma'$ there exists a short exact sequence of $G_\p$-modules
\begin{align}\label{smext}
0\too \St_\p \too \mathfrak{E}_{\p,\sm}\too \Z\too 0
\end{align}
with the $\Z[G_\p]$-module $\mathfrak{E}_{\p,\sm}$ having a resolution of the form
\begin{align}\label{smextres}
0\too C_c(\scr{V}_\p,\Z)\too C_c(\scr{V}_\p,\Z) \too \mathfrak{E}_{\p,\sm}\too 0
\end{align}
(see \cite{Sp}, Equation (17)).
By analyzing the long exact sequence in cohomology induced by \eqref{smext} we see that it is enough to show that
\[
\HH^{d}_\Z\big(X_G^\Sigma(\mathfrak{f}), \St_{\Sigma'\setminus\{\p\}}(\chi_{\Sigma\setminus (\Sigma'\setminus\{\p\})}\big) \otimes_\Z \mathfrak{E}_{\p,\sm}, \Z)_{\pi}^{\epsilon}
\]
is torsion for all $d\geq 0.$
This can be deduced from \eqref{smextres} and the last claim of Proposition \ref{arithclass}.
\end{proof}

\noindent The same ideas as those in the proof of Proposition \ref{class} prove the following useful lemma.
\begin{lemma}\label{bijremark}
Let $\Sigma_1\subseteq \Sigma_2\subseteq \Sigma$ be subsets of $S$.
The map
\begin{align*}
\Ev_{\Sigma_1}^*\colon\HH^{d}_\Z\big(X_G^\Sigma(\mathfrak{f}),\St_{\Sigma_2}(\chi_{\Sigma\setminus\Sigma_2}),\Z\big)_{\pi}^{\epsilon}
\too \HH^{d}_\Z\big(X_G^{\Sigma\setminus \Sigma_1}(\mathfrak{f}),\St_{\Sigma_2\setminus \Sigma_1}(\chi_{\Sigma\setminus\Sigma_2}),\Z\big)_{\pi}^{\epsilon},
\end{align*}
induced by $\Ev_{\Sigma_1}\colon  \bigotimes_{\p\in \Sigma_1} C_c(\scr{E}_\p,\Z)\too \St_{\Sigma_1}$
has finite kernel and cokernel for all $d\geq 0$.
\end{lemma}

\begin{definition}\label{genelliptic}
Let $\Sigma\subseteq S$ be a subset.
We define $ c_{A,\epsilon}^{\Sigma}$ to be any generator of the free part of $\HH^{q}(X_G^\Sigma(\mathfrak{f}),\St_{\Sigma},\Z)_{\pi}^{\epsilon}$, which is unique up to sign.
For convenience, we also write $c_{A,\epsilon}^{\Sigma}$ for the image of the generator under the natural map
\[
\HH^{q}_\Z\big(X_G^\Sigma(\mathfrak{f}),\St_{\Sigma},\Z\big)\too \HH^{q}_{\Z_p}\big(X_G^\Sigma(\mathfrak{f}),\St_{\Sigma}(\Z_p),\Z_p\big).
\]
\end{definition}


\section{Local points on modular elliptic curves}
From here on, we only work with $\Z_p$-modules and in level $\mathfrak{f}$.
Thus, we shorten the notation to
\[
\HH^{q}(X_G^\Sigma,M,N)=\HH^{q}_{\Z_p}(X_G^\Sigma(\frak{f}),M,N)
\]
for every $\Z_p$-module $N$ and every $\Z_p[G_\Sigma]$-module $M$.

\subsection{p-adic integration}\label{Integration}
Let $\p$ be an element in $S$ and $M$ a finitely generated $\Z_p$-module. The $M$-valued continuous $\pm$-Steinberg representation 
 $\St^{\pm}_\p(M)^{\cont}$ of  $G_\p$
 is the space of continuous $M$-valued functions on $\PP(F_\p)$ modulo constant functions, with action given by 
 \[
 (g.f)(x)=\chi_\p^\pm(g)\cdot f(g^{-1}x)\qquad\forall\ g\in G_\p,\ x\in \PP(F_\p).
 \]	
 It is equipped with a canonical $G_\p$-equivariant injection $\St_\p^\pm(M) \into \St^{\pm}_\p(M)^{\cont}$.

\begin{definition}
	For the fixed quadratic extension $E/F$ we define
	\[
	\St_{\Sigma}^{\cont}(\widehat{E}^\times):=\bigotimes_{\p\in\Sigma^+}\St_\p^+(\widehat{E}^\times_\p)^{\cont} \otimes_{\bb{Z}_p} \bigotimes_{\p\in\Sigma^-} \St_\p^-(\widehat{E}^\times_\p)^{\cont}
	\]
	where $\widehat{E}^\times_\p$ denotes the torsion-free part of the $p$-adic completion of $E_\p^\times$. Moreover, we denote by  $\widehat{E}^\times_{\Sigma,\otimes}:=\bigotimes_{\mathfrak{p}\in\Sigma}\widehat{E}^\times_\p$ the tensor product over $\bb{Z}_p$ and write 
	\[
	\St_{\Sigma}(\widehat{E}^\times):=\St_{\Sigma}\otimes_{\bb{Z}}\widehat{E}^\times_{\Sigma,\otimes}.
	\]
\end{definition}

\begin{lemma}\label{restrictioniso}
Let $N$ be a $p$-adically separated and complete $\bb{Z}_p$-module.
Then, for any subset $\Sigma\subseteq S$ the restriction map
\[
\Hom_{\Z_p}\big(\St_{\Sigma}^{\cont}(\widehat{E}^\times),N\big) \too \Hom_{\Z_p}\big(\St_{\Sigma}(\widehat{E}^\times),N\big)
\]
is an isomorphism.
\end{lemma}
\begin{proof}
Any continuous function in $C(\PP(F_\p),\widehat{E}_\p^\times)$ is a $p$-adic limit of locally constant function in $\mrm{LC}(\PP(F_\p),\widehat{E}_\p^\times)$. Therefore, the $\bb{Z}_p$-tensor product of spaces of locally constant functions is $p$-adically dense in the $\bb{Z}_p$-tensor product of spaces of continuous functions
\[
\bigotimes_{\p\in\Sigma}\mrm{LC}(\PP(F_\p),\widehat{E}_\p^\times)\hookrightarrow \bigotimes_{\p\in\Sigma}C(\PP(F_\p),\widehat{E}_\p^\times).
\]
As any $\bb{Z}_p$-linear homomorphism is continuous for the $p$-adic topology and $N$ is $p$-adically separated and complete, we deduce that the induced map
\[
\Hom_{\Z_p}\big(\bigotimes_{\p\in\Sigma}C(\PP(F_\p),\widehat{E}_\p^\times),N\big) \too \Hom_{\Z_p}\big(\bigotimes_{\p\in\Sigma}\mrm{LC}(\PP(F_\p),\widehat{E}_\p^\times),N\big)
\]
is an isomorphism.  
\end{proof}

\noindent Given a $p$-adically separated and complete $\bb{Z}_p$-module $N$, Lemma $\ref{restrictioniso}$ yields an integration map
\begin{equation}\label{integrate}
	\mint\colon  \Hom_{\Z_p}\big(\St_{\Sigma}(\bb{Z}_p),N\big)\too \Hom_{\Z_p}\big(\St_{\Sigma}^{\cont}(\widehat{E}^\times),N\otimes_{\Z_p} \widehat{E}^\times_{\Sigma,\otimes}\big)
\end{equation}
that takes an $N$-valued measure $\mu\in \Hom_{\Z_p}\big(\St_{\Sigma}(\bb{Z}_p),N\big)$ and maps it to the linear functional 
\[
\mint_{\bb{P}^1(F_\Sigma)}(-) \mrm{d}\mu\ \in \Hom_{\Z_p}\big(\St_{\Sigma}^{\cont}(\widehat{E}^\times),N\otimes_{\Z_p} \widehat{E}^\times_{\Sigma,\otimes}\big).
\]

\subsubsection{The $p$-adic upper half plane.}
We recall that any $\p\in S$ is inert in the quadratic extension $E/F$.
Hence,  we can consider the $G_\p$-space $\mathcal{H}_\p(E_\p)=\PP(E_\p)\setminus\PP(F_\p)$
 where the group $G_\p$ acts via M\"obius transformations. We denote by $\Div_{0}^\pm(\mathcal{H}_\p(E_\p))$ the $G_\p$-module obtained by twisting the action on degree zero divisors by the character $\chi_\p^\pm$.

\noindent We consider the $G_\p$-equivariant homomorphism
\[
\Phi_\p^\pm\colon \Div^\pm_{0}\big(\mathcal{H}_\p(E_\p)\big)\too \St_\p^\pm\big(\widehat{E}_\p^{\times}\big)^{\cont}
\]
mapping a degree $0$ divisor $\sum_{i=1}^kn_i [z_i]$ to the class of functions represented by $\prod_{i=1}^k(x-z_i)^{n_i}$.
The embedding $\psi\colon T\hookrightarrow G$ induces an action of $T_\p$
 on $\mathcal{H}_\p(E_\p)$ with two fixed points $\tau_{\psi,\p}$ and $\bar{\tau}_{\psi,\p}$ which are interchanged by the Galois group $\Gal(E_\p/F_\p)=\langle\sigma_\p\rangle$. 
 We choose $\tau_{\psi,\p}\in \mathcal{H}_\p(E_\p)$ such that the action of $T_\p$ on the tangent space of $\tau_{\psi,\p}$ is given by the homomorphism
 \[
 T_\p\too E_\p^{\times},\quad t\mapsto t^{1-\sigma_\p}.
 \]
 Since the character $\chi_\p^{\pm}$ is trivial on $T_\p$, the are $T_\p$-equivariant maps
 \[
  	\Psi_{\p}^{\diamond}\colon \Z_p \too \St^\pm_\p(\widehat{E}_\p^{\times})^{\cont},\quad 1\mapsto \Phi^\pm_\p([\tau_{\psi,\p}]-[\bar{\tau}_{\psi,\p}])
 \]
that can be combined into a $T_S$-equivariant homomorphism
 \[
\Psi^{\diamond}_{S}\colon \Z_p \too \St^{\cont}_S\big(\widehat{E}^{\times}\big), \quad 1\mapsto \bigotimes_{\mathfrak{p}\in S}\Phi_\p^\pm\big([\tau_{\psi,\p}]-[\bar{\tau}_{\psi,\p}]\big).
 \]
The composition of $\Psi^{\diamond}_{S}$ with \eqref{integrate}, yields the $T_S$-equivariant map
\begin{equation}\label{Xi}
	(\Psi^{\diamond}_S)^{\ast}\colon \Hom_{\Z_p}\big(\St_S(\Z_p),\Z_p\big)\too\widehat{E}_{S, \otimes}^{\times},
\end{equation}
where $T_S$ acts trivially on $\widehat{E}_{S, \otimes}^{\times}$. Concretely, a $\bb{Z}_p$-valued measure $\mu\in \Hom_{\Z_p}\big(\St_{S}(\bb{Z}_p),\bb{Z}_p\big)$ is mapped to the integral
\[
(\Psi^{\diamond}_{S})^{\ast}(\mu)=\mint_{\bb{P}^1(F_S)}\bigotimes_{\p\in S}\bigg(\frac{t_\p-\tau_{\psi,\p}}{t_\p-\bar{\tau}_{\psi,\p}}\bigg) \mrm{d}\mu(t)\ \in  \widehat{E}^\times_{S,\otimes}.
\]

\subsection{Plectic p-adic invariants}\label{Normalized}
For any non-zero ideal $\mathfrak{m}\subseteq \OO_F$, any abelian group $N$, and $?\in\{\emptyset, c\}$ we can consider the space of functions
$C_{?}(\mathfrak{m},N):=C_{?}(T(\A^{\infty})/U(\mathfrak{m}),N)$, and set
 \[ \mrm{H}^d(X_T(\mathfrak{m}),N):=\mrm{H}^d(T(F)^+,C(\mathfrak{m},N))\qquad\&\qquad
\mrm{H}_d(X_T(\mathfrak{m}),N):=\mrm{H}_d(T(F)^+,C_c(\mathfrak{m},N)).
\]
By Dirichlet's unit theorem, the abelian group $\mathcal{U}^+=T(F)^+\cap U(\OO_F)$ is finitely generated of rank $q$.
Thus, the homology group $\HH_{q}(\mathcal{U}^+,\Z)$ is a finitely generated abelian group of rank one.
We fix an element $\eta\in\HH_{q}(\mathcal{U}^+,\Z)$ generating the free part of that group.
Further, by fixing a fundamental domain $\mathcal{F}$ for the action of $T(F)^+/\mathcal{U}^+$ on $T(\A^{\infty})/U(\OO_F)$, Shapiro's lemma gives
\[
\HH_{q}(\mathcal{U}^{+},C(\mathcal{F},\Z))\xlongrightarrow{\cong} \mrm{H}_{q}(X_T(\OO_F),\Z).
\]

\begin{definition}
The fundamental class $\vartheta\in \mrm{H}_{q}(X_T(\OO_F),\Z)$ is defined as the image of the cap product of $\eta\in \HH_{q}(\mathcal{U}^+,\Z)$ with the characteristic function $\cf_{\mathcal{F}}\in \HH^{0}(\mathcal{U}^+,C(\mathcal{F},\Z))$ under the above isomorphism.
\end{definition}

\subsubsection{Twisted fundamental classes.}\label{twistedfun}
Remember that the preimage $K(\mathfrak{f})^S$ under $\psi$ is of the form $U(\mathfrak{c})^S$ for a non-zero ideal $\mathfrak{c}\subseteq \OO_F$ coprime to $\disc(B)\cdot\mathfrak{f}$.
Let $E_\mathfrak{c}/E$ be the narrow ring class field of $E$ of conductor $\mathfrak{c}$,
i.e.~the Galois extension of $E$ such that the Artin map induces an isomorphism
\[
\mrm{rec}_E\colon T(F)^{+}\backslash T(\A^{\infty})/U(\mathfrak{c})\xlongrightarrow{\sim}\Gal(E_\mathfrak{c}/E)=:\mathcal{G}_{\mathfrak{c}}.
\]
Let $R$ be a ring and $\chi\colon \mathcal{G}_{\mathfrak{c}}\to R^{\times}$ a group homomorphism, then we are allowed to view $\chi$ as an element in $\mrm{H}^{0}(X_T(\mathfrak{c}),R).$
Note that multiplication of functions
\[
C(\mathfrak{c},R)\times C_c(\OO_F,\Z)\too C_c(\mathfrak{c},R),\qquad (f, g)\mapsto  f\cdot g
\]
induces the cap product pairing
\begin{align*}
\cap\colon \mrm{H}^{0}(X_T(\mathfrak{c}),R)\times \mrm{H}_{q}(X_T(\OO_F),\Z)\too \mrm{H}_{q}(X_T(\mathfrak{c}),R).
\end{align*}
\begin{definition}\label{twisted}
The $\chi$-twisted fundamental class is defined as
\[\
\vartheta_{\chi}:=\chi\cap\vartheta\in \mrm{H}_{q}(X_T(\mathfrak{c}),R).
\]
\end{definition}
\begin{remark}\label{characterinf}
	We may view a character $\chi\colon \cal{G}_\mathfrak{c}\to\overline{\bb{Q}}^\times$ as a locally constant character $T(\A)\to \overline{\bb{Q}}^\times$ via the Artin reciprocity map.
	The infinity component $\chi_\infty\colon T_\infty \to \{\pm 1\}$ of $\chi$ factors through the quotient $T_\infty\to\pi_0(T_\infty).$
	The inclusion $T(F)\into T_\infty$ induces an isomorphism $T(F)/T(F)^+\cong\pi_0(T_\infty).$
	Therefore, we may view $\chi_\infty$ as a character on $T(F)/T(F)^+.$
	The group $T(F)/T(F)^+$ naturally acts on $\mrm{H}_{q}(X_T(\mathfrak{c}),\overline{\Q})$ and it is easy to see that 
	\[
	\vartheta_{\chi}\in \mrm{H}_{q}(X_T(\mathfrak{c}),\overline{\Q})^{\chi_\infty},
	\]
	where the subscript denotes the $\chi_\infty$-isotypic component of the $T(F)/T(F)^+$-action.
	The inclusion $T_\infty\into G_\infty$ induces an isomorphism of component groups $\pi_0(T_\infty)\xrightarrow{\sim}\pi_0(G_\infty)$ and we always identify the two.
\end{remark}

\subsubsection{Plectic invariants.}
For any ring $R$ and abelian group $N$ the canonical pairing 
\[
C(\mathfrak{c},N)\times C_c(\mathfrak{c},R)\too N\otimes_\Z R,\qquad
(f, g)\mapsto \sum_{t\in T(\A^{\infty})/U(\mathfrak{c})}\hspace{-1em} f(t)\otimes g(t)
\]
induces a cap product pairing $
\cap\colon \mrm{H}^{q}(X_T(\mathfrak{c}),N)\times \mrm{H}_{q}(X_T(\mathfrak{c}),R)\too N\otimes_\Z R$.

\noindent Since every $\p\in S$ is inert in $E$, and $\mathfrak{c}$ is coprime to every $\p\in\Sigma$ we may identify
$$C(T(\A^{S,\infty})/U(\mathfrak{c})^{S},N)=C(T(\A^{\infty})/U(\mathfrak{c}),N)$$
for every abelian group $N$.
Suppose we are given a $\Z_p$-module $N$, a $\Z_p[G_S]$-module $M_G$ -- viewed as a $T_S$-module via $\psi$ -- and a $T_S$-equivariant map $f\colon \Z_p \to M_G$.
Then, the inclusion $\psi$ and the pullback along $f$ induce the map
\[C(G(\A^{S,\infty})/K(\mathfrak{f})^{S},\Hom_{\Z_p}(M_G,N))\too C(T(\A^{\infty})/U(\mathfrak{c}),N),
\]
which in turn induces the homomorphism in cohomology
\[
f^{\ast}\colon \HH^{d}(X_G^S,M_G,N)\too\HH^d(X_T(\mathfrak{c}),N).
\]
Applying this construction to \eqref{Xi} yields the homomorphism
\[
(\Psi^{\diamond}_{S})^{\ast}\colon\HH^{q}(X_G^S,\St_{S}(\Z_p),\Z_p)\too \HH^q(X_T(\mathfrak{c}),\widehat{E}_{S,\otimes}^{\times}).
\]

\begin{definition}
Let $\chi\colon \mathcal{G}_{\mathfrak{c}}\to \overline{\bb{Q}}^\times$ be a character with $\chi_\infty=\epsilon$. The \emph{plectic $p$-adic invariant} attached to the pair $(A_{/F},\chi)$ is 
\[
\mrm{Q}_{A}^{\chi}:=(\Psi^{\diamond}_S)^{\ast}( c_{A,\epsilon}^S)\cap \vartheta_{\chi}\ \in\  \widehat{E}_{S,\otimes}^{\times}\otimes_\Z\overline{\bb{Q}}.
\]
It is uniquely defined up to sign.
\end{definition}

\noindent Note that if $\chi_\infty\not =\epsilon$, then by the orthogonality of characters and Remark \ref{characterinf} we have 
\[
(\Psi^{\diamond}_S)^{\ast}( c_{A,\epsilon}^S)\cap \vartheta_{\chi}=0.
\]

\subsubsection{Tate uniformization.}
The elliptic curve $A$ has multiplicative reduction at every $\p \in S$.
Therefore, it admits a Tate uniformization $E_\p^{\times}\to A(E_\p)$ and correspondingly a map $\widehat{E}_\p^{\times} \to \widehat{A}(E_\p)$
between the torsion-free parts of the $p$-adic completions.
We form the map
\[
\phi_{\mbox{\tiny $\mrm{Tate}$}} \colon  \widehat{E}_{S,\otimes}^{\times}\too \widehat{A}(E_S)
\]
where $\widehat{A}(E_S):=\bigotimes_{\p\in S}\widehat{A}(E_{\p})$ is a tensor product of $\bb{Z}_p$-modules.
Then, we can consider
\[
\phi_{\mbox{\tiny $\mrm{Tate}$}}\big(\mrm{Q}_{A}^{\chi}\big)\ \in\ \widehat{A}(E_{S})\otimes_\Z \overline{\bb{Q}}.
\]
The goal of the next sections is to refine the construction to obtain plectic Stark--Heegner points $\mrm{P}_{A}^{\chi} \in \widehat{A}(E_{S})\otimes_\Z\overline{\Q}.$

\subsection{Breuil's extensions}\label{Extensions}
We recall the extensions of Steinberg representations constructed by Breuil (\cite{Br}, Section 2.1) and Spie\ss~(\cite{Sp}, Section 3.7) together with their connection to divisors on the $p$-adic upper half plane following (\cite{BG2}, Section 6.3). 
Let $M$ a free $\Z_p$-module of finite rank.
Given a continuous group homomorphism $\ell\colon F_\p^\times\to M$, one defines $\widetilde{\mathfrak{E}}_\p(\ell)$ as the collection of pairs $(\Phi,r)\in C(\GL_2(F_\p),M)\times \Z_p$ satisfying
\[
\Phi\left(g\cdot\begin{pmatrix}
t_1 & u \\ 0 & t_2
\end{pmatrix}\right)
=\Phi(g) + r\cdot \ell(t_1)
\]
for all $t_1,t_2\in F_\p^\times$, $u\in F_\p$ and $g\in\GL_2(F_\p)$.
The group $\GL_2(F_\p)$ acts on $\widetilde{\mathfrak{E}}_\p(\ell)$ by \[
g.(\Phi(\cdot),r)=(\Phi(g^{-1}\cdot),r).
\]
The subspace $\widetilde{\mathfrak{E}}_\p(\ell)_0$ of tuples of the form $(\Phi,0)$ with $\Phi$ a constant function is $\GL_2(F_\p)$-stable.
We deduce that the quotient $\mathfrak{E}_\p(\ell)=\widetilde{\mathfrak{E}}_\p(\ell) /\widetilde{\mathfrak{E}}_\p(\ell)_0$ admits an induced $G_\p\cong\PGL_2(F_\p)$-action.
Let $\pr\colon \PGL_{2}(F_\p)\to \PP(F_\p)$, $g\mapsto g(\infty)$ be the canonical projection.
The sequence
\[	
0\too\St^{\cont}_{\p}(M)\xlongrightarrow{(\pr^{*},0)} \mathfrak{E}_\p(\ell)\xlongrightarrow{(0,\id_{\Z_p})} \Z_p\too 0.
\]
of $G_\p$-modules is exact (see \cite{Sp}, Lemma 3.11).
We denote by $b_{\ell,\p}\in \HH^{1}\big(G_\p,\St^{\cont}_{\p}(M)\big)$ the class of the extension of $\Z_p[G_\p]$-modules.
	The next lemma explains the relation between Breuil's extensions and the extension of divisors on the $p$-adic upper half plane
\[
0\too \Div^0(\mathcal{H}_\p(E_\p))\too \Div(\mathcal{H}_\p(E_\p))\xlongrightarrow{\deg} \Z\too 0,
\]
whose class we denote by $b_{\mrm{Div},\p}\in \HH^1\big(G_\p,\Div^{0}(\mathcal{H}_\p(E_\p))\big)$.

\begin{lemma}\label{compextensions}
Let $\iota\colon F_\p^{\times}\to \widehat{E}_\p^{\times}$ be the homomorphism induced by the inclusion $F_\p^{\times}\into E_\p^{\times}$. Then 
\[
b_{\iota,\p}=(\Phi_{\p})_{\ast}(b_{\mrm{Div},\p})
\]
holds in $\HH^{1}\big(G_\p,\St^{\cont}_{\p}(\widehat{E}_\p^{\times})\big)$.
\end{lemma}
\begin{proof}
This is (\cite{BG2}, Lemma 6.8).
To prove the claim we need to construct a $G_\p$-equivariant homomorphism $\widetilde{\Phi}_\p\colon \Div(\mathcal{H}_\p(E_\p)) \to \mathfrak{E}_\p(\iota)$ that is compatible with $\Phi_\p.$
For $z\in \mathcal{H}_\p(E_\p)$ we define
\[
f_z\colon \GL_2(F_\p)\too \widehat{E}_\p^{\times},\quad \begin{pmatrix}a & b\\ c & d\end{pmatrix}\mapstoo cz-a.
\]
Then $\widetilde{\Phi}_\p$ is given by
$$\widetilde{\Phi}_\p\left(\sum_{z} n_z [z]\right)=\left(\prod_{z} f_z^{n_z}, \sum_{z}n_z\right).$$
\end{proof}

\noindent We put $\mathfrak{E}_\p=\mathfrak{E}_\p(\iota)$ and we let $\tau_{\psi,\p}$ be the chosen fixed point of the $T_\p$-action on $\mathcal{H}_\p(E_\p)$.
The proof of Lemma $\ref{compextensions}$ provides a $T_\p$-equivariant map 
\[
\Psi_{\p}\colon \Z_p \too \mathfrak{E}_\p,\quad 1\mapsto \widetilde{\Phi}_\p([\tau_{\psi,\p}],1).
\]
The choice of the other fixed point amounts to acting on $\widehat{E}_\p^{\times}$ by the non-trivial element $\sigma_\p$ of $\Gal(E_\p/F_\p)$.
Directly from the definitions we see that the following equation holds in $\mathfrak{E}_\p$
\begin{equation}\label{relation}
	\Psi_{\p}-\sigma_\p^{\ast}\circ\Psi_{\p}= \Psi_{\p}^\diamond.
\end{equation}

\begin{definition}\label{BreuilExtension}
	We set $\mathfrak{E}_\p^\pm=\mathfrak{E}_\p(\chi_\p^{\mbox{\tiny $\pm$}})$ and for any subset $\Sigma\subseteq S$ we define
	\[
	\mathfrak{E}_\Sigma=\bigotimes_{\p\in\Sigma^+} \mathfrak{E}_\p^+ \otimes_{\Z_p}\bigotimes_{\p\in\Sigma^{-}} \mathfrak{E}_\p^{-}.
	\]
\end{definition}
\noindent Then, we can form the $T_S$-equivariant map
 \[
\Psi_S\colon \Z_p \too \mathfrak{E}_S, \quad 1\mapsto \bigotimes_{\mathfrak{p}\in S}\widetilde{\Phi}_\p\big([\tau_{\psi,\p}],1\big),
\]
inducing a $T_S$-equivariant homomorphism
\begin{equation}\label{Xi2}
	(\Psi_{S})^{\ast}\colon \Hom_{\Z_p}\big(\mathfrak{E}_S,\widehat{A}(E_S)\big)\too\widehat{A}(E_S).
\end{equation}

\subsection{Plectic Stark--Heegner points}\label{Points}
For this section and the next, we assume the equality of arithmetic and automorphic $\LI$-invariants of modular elliptic curves over $F$. We note that when $F$ is totally real, the equality is known (\cite{GeR}, \cite{Sp3}).
This equality is a crucial ingredient (see the proof of Theorem \ref{Step1} for a precise statement) in lifting the image of $c^S_{A,\epsilon}$ under the composition
\[\xymatrix{
\mrm{H}^q\big(X_G^S,\St_S,\Z_p\big)\ar[r]  &   \mrm{H}^q\big(X_G^S,\St^{\cont}_S(\widehat{E}^{\times}),\widehat{E}_{S,\otimes}^{\times}\big)\ar[r]  & \mrm{H}^q\big(X_G^S,\St^{\cont}_S(\widehat{E}^{\times}),\widehat{A}(E_{S})\big),
}\]
with respect to the map induced in cohomology by the inclusion $\St^{\cont}_S(\widehat{E}^{\times}) \hookrightarrow\mathfrak{E}_S$
\[
\mrm{H}^q\big(X_G^S,\frak{E}_S,\widehat{A}(E_{S})\big) \too \mrm{H}^q\big(X_G^S,\St^{\cont}_S(\widehat{E}^{\times}),\widehat{A}(E_{S})\big).
\]

\begin{theorem}\label{LIinv}
There exists a multiple of $c_{A,\epsilon}^{S}$ that can be lifted to a class
\[
\widehat{c}_{A,\epsilon}^{S}\in \mrm{H}^q\big(X_G^{S},\mathfrak{E}_S,\widehat{A}(E_{S})\big)_{\pi}^{\epsilon}.
\]
Moreover, the kernel of
		\[
		\mrm{H}^q\big(X_G^{S},\mathfrak{E}_S,\widehat{A}(E_{S})\big)_{\pi}^{\epsilon} \too \mrm{H}^q\big(X_G^{S},\St^{\cont}_S(\widehat{E}^{\times}),\widehat{A}(E_{S})\big)_{\pi}^{\epsilon}
		\] is finite.
\end{theorem}

\noindent As the proof of Theorem \ref{LIinv} is rather technical, we postpone it until the next section.
Let us just say that the proof proceeds by induction on the size of $\Sigma\subseteq S$, and that the base case follows directly from the equality of $\cal{L}$-invariants.
Moreover, the second claim of Theorem \ref{LIinv} is a special case of Lemma \ref{inj} below.

\noindent Momentarily assuming the existence of the lift, we can proceed as in the definition of plectic invariants since the restriction map together with \eqref{Xi2} yields a homomorphism
\[
(\Psi_{S})^{\ast}\colon \mrm{H}^q\big(X_G^{S},\mathfrak{E}_S,\widehat{A}(E_{S})\big)\too \mrm{H}^q\big(X_T(\mathfrak{c}),\widehat{A}(E_S)\big).
\]

\begin{definition}
Let $\chi\colon \mathcal{G}_{\mathfrak{c}}\to \overline{\bb{Q}}^\times$ be a character with $\chi_\infty=\epsilon$.
The \emph{plectic Stark--Heegner point} attached to the pair $(A_{/F},\chi)$  is 
\[
\mrm{P}_{A}^{\chi}:=(\Psi_{S})^{\ast}(\widehat{ c}^{S}_{A,\epsilon})\cap \theta_{\chi}\in \widehat{A}(E_{S})\otimes_\Z \overline{\Q}.
\]
It does not depend on the choice of the lift $\widehat{c}_{A,\epsilon}^S$.
\end{definition}

\noindent A direct consequence of \eqref{relation}, and the Galois equivariance of the Tate uniformization is 
\[
\bigg(\prod_{\p\in S^+}(1-\sigma_\p^{\ast})\prod_{\p\in S^{-}}(1+\sigma_\p^{\ast})\bigg) \mrm{P}_{A}^{\chi}=  \phi_{\mbox{\tiny $\mrm{Tate}$}} \big(\mrm{Q}_{A}^{\chi}\big),
\]
the equality  in $\widehat{A}(E_{S})\otimes_\Z \overline{\Q}$.

\subsection{Lifting the cohomology class}\label{Appendix} 
To lighten the exposition, in this section we assume that the elliptic curve $A_{/F}$ has split multiplicative reduction at all primes in $S$. This way the character $\chi_\Sigma$ is trivial for any subset $\Sigma\subseteq S$.

\noindent The following notations will be convenient in the proofs.
Given a subset $\Sigma\subseteq S$, and a collection $M=(M_\p)_{\p\in \Sigma}$ of free $\Z_p$-modules of finite rank, we put 
\[
\St_{\Sigma}(M):=\bigotimes_{\p\in\Sigma}\St_\p(M)\qquad\&\qquad
\St_{\Sigma}^{\cont}(M):=\bigotimes_{\p\in\Sigma}\St_\p(M)^{\cont}.
\]
We write $\Z_p$ for the constant collection $(\Z_p)_{\p\in\Sigma}$.
Moreover, Lemma \ref{restrictioniso} extends to this context, i.e. the map induced by the inclusion $\St_{\Sigma}(M)\hookrightarrow \St_{\Sigma}^{\cont}(M)$
\[
\Hom_{\Z_p}\big(\St_{\Sigma}^{\cont}(M)\otimes_{\Z_p} N',N\big) \overset{\sim}{\too} \Hom_{\Z_p}\big(\St_{\Sigma}(M)\otimes_{\Z_p} N',N\big)
\]
is an isomorphism for a $p$-adically separated and complete $\Z_p$-module $N$ and a $\Z_p$-module $N'$.

\noindent We begin with a lemma establishing the uniquess, up to finite torsion,  of the lift, if it exists.
\begin{lemma}\label{inj}
Let $\Sigma_1 \subseteq \Sigma_2 \subseteq \Sigma$ be subsets of $S$ and $N$ a finitely generated $\Z_p$-module.
For any collection $M=(M_\p)_{\p\in\Sigma_2}$ of finite free $\Z_p$-modules with $M_\p=\widehat{E}^{\times}_{\p}$ for all $\p\in\Sigma_1$, the kernel of the $\Z_p$-module homomorphism
\[
\mrm{H}^{q+\lvert\Sigma\setminus\Sigma_2\rvert}\big(X_G^{\Sigma},\mathfrak{E}_{\Sigma_1}\otimes_{\Z_p}\St^{\cont}_{\Sigma_2\setminus \Sigma_1}(M),N\big)_{\pi}^{\epsilon}\too
\mrm{H}^{q+\lvert\Sigma\setminus\Sigma_2\rvert}\big(X_G^{\Sigma},\St^{\cont}_{\Sigma_2}(M),N\big)_{\pi}^{\epsilon},
\]
induced by the inclusion $\St^{\cont}_{\Sigma_1}\big(\widehat{E}^{\times}\big) \into\mathfrak{E}_{\Sigma_1}$, is finite.
\end{lemma}
\begin{proof}
We prove the following slightly more general statement:
let $\Sigma_0 \subseteq \Sigma_1 \subseteq \Sigma_2 \subseteq \Sigma$ be subsets of $S$, the kernel of the $\Z_p$-module homomorphism
\[
\mrm{H}^{q+\lvert\Sigma\setminus\Sigma_2\rvert}\big(X_G^{\Sigma},\mathfrak{E}_{\Sigma_1}\otimes_{\Z_p}\St^{\cont}_{\Sigma_2\setminus \Sigma_1}(M),N\big)_{\pi}^{\epsilon}\too
\mrm{H}^{q+\lvert\Sigma\setminus\Sigma_2\rvert}\big(X_G^{\Sigma},\mathfrak{E}_{\Sigma_0}\otimes_{\Z_p}\St^{\cont}_{\Sigma_2\setminus \Sigma_0}(M),N\big)_{\pi}^{\epsilon},
\]
induced by the inclusion $\St^{\cont}_{\Sigma_1\setminus \Sigma_0}\big(\widehat{E}^{\times}\big) \into\mathfrak{E}_{\Sigma_1\setminus \Sigma_0}$, is finite.
Clearly, it is enough to consider the case $\Sigma_0=\Sigma_1\setminus\left\{\p\right\}$ for some $\p\in \Sigma_1.$

\noindent By analyzing the long exact sequence induced by the short exact sequence defining $\mathfrak{E}_\p$, we deduce that it suffices to prove that 
\[
\mrm{H}^{q+\lvert\Sigma\setminus\Sigma_2\rvert}\big(X_G^{\Sigma};\mathfrak{E}_{\Sigma_0}\otimes_{\Z_p}\St^{\cont}_{\Sigma_2\setminus \Sigma_1}(M),N\big)_{\pi}^{\epsilon}
\]
is finite.
Moreover, the long exact sequences induced by the short exact sequences defining $\mathfrak{E}_\q$ for every $\q\in \Sigma_0$, implies that it is enough to prove that 
\[
\mrm{H}^{q+\lvert\Sigma\setminus\Sigma_2\rvert}\big(X_G^{\Sigma};\St_{(\Sigma_2\setminus\Sigma_1)\cup\Sigma'}(M),N\big)_{\pi}^{\epsilon}
\]
is finite for every subset $\Sigma'$ of $\Sigma_0$.
This last claim follows from flat base change and Proposition \ref{class} because $(\Sigma_2\setminus\Sigma_1)\cup\Sigma'$ is a proper subset of $\Sigma_2$.
\end{proof}

\noindent Next, we deal with the existence of the lift. The following theorem relies on the equality of arithmetic and automorphic $\cal{L}$-invariants.
\begin{theorem}\label{Step1}
	Suppose $\Sigma=\{\p\}$ and consider the boundary map 
\[
\partial_\p\colon \mrm{H}^q\big(X_G^\Sigma,\St^{\cont}_\p(\widehat{E}_\p^{\times}),\widehat{A}(E_\p)\big) \too \mrm{H}^{q+1}\big(X_G^\Sigma,\Z_p,\widehat{A}(E_\p)\big)
\]
obtained from the short exact sequence defining $\mathfrak{E}_\p$. Then, the class $\partial_\p\big(c_{A,\epsilon}^{\{\p\}}\big)$ is torsion.
\end{theorem}
\begin{proof}
By the equality of automorphic and arithmetic $\LI$-invariants (\cite{GeR}, Theorem 4.1 $\&$  \cite{Sp3}, Theorem 3.7) the image  of $\mrm{H}^q\big(X_G^\Sigma,\St_\p(\Z_p),\Z_p\big)_{\pi}^{\epsilon}$
 under
\[\partial_\p\colon \mrm{H}^q\big(X_G^\Sigma,\St^{\cont}_\p(\widehat{E}_\p^{\times}),\widehat{E}_\p^{\times}\big)
 \too \mrm{H}^{q+1}\big(X_G^\Sigma,\Z_p,\widehat{E}_\p^{\times}\big)
\]
 is commensurable with the image of the natural map
 \[ \mrm{H}^{q+1}\big(X_G^\Sigma,\Z_p,\Z_p q_\p^\mrm{Tate}\big)\to \mrm{H}^{q+1}\big(X_G^\Sigma,\Z_p,\widehat{E}_\p^{\times}\big)
 \]
 where $q_\mathfrak{p}^\mrm{Tate}$ is Tate's period for $A$ at $\p$.
Therefore, the cokernel of
\[
\mrm{H}^q\big(X_G^\Sigma,\St_\p(\Z_p),\Z_p\big)_{\pi}^{\epsilon}\too \mrm{H}^{q+1}\big(X_G^\Sigma,\Z_p, \widehat{A}(E_\p) \big)
\]
 is torsion and the claim follows.
\end{proof}

\begin{remark}
When $B$ is totally definite, the equality of automorphic and arithmetic $\LI$-invariants can be deduced from Manin and Drinfeld's uniformization of Jacobians of Mumford curves (cf.~\cite{BG2}, Theorem 6.9).
\end{remark}

\noindent The following lemma follows directly from Lemma \ref{bijremark} analyzing the long exact sequence induced by the short exact sequence defining the representations $\mathfrak{E}_{\p}$.
\begin{lemma}\label{bijlemma}
	Let $\Sigma_1 \subseteq \Sigma_2 \subseteq \Sigma$ be subsets of $S$, $M=(M_\p)_\p$ a collection of finite free $\Z_p$-modules such that $M_\p=\Z_p$ for all $\p \in \Sigma_1$, and $N$ a finitely generated $\Z_p$-module.
	For all $d\geq 0$, kernel and cokernel of the $\Z_p$-module homomorphism
	\[
	(\Ev_{\Sigma_1})^*\colon \mrm{H}^{d}\big(X_G^\Sigma,\mathfrak{E}_{\Sigma\setminus\Sigma_2}\otimes_{\Z_p}\St_{\Sigma_2}(M),N\big)_{\pi}^{\epsilon}\too
	\mrm{H}^{d}\big(X_G^{\Sigma\setminus\Sigma_1},\mathfrak{E}_{\Sigma\setminus\Sigma_2}\otimes_{\Z_p}\St_{\Sigma_2\setminus \Sigma_1}(M),N\big)_{\pi}^{\epsilon},
	\]
	induced by $\Ev_{\Sigma_1}\colon  \bigotimes_{\p\in \Sigma_1} C_c(\scr{E}_\p,\bb{Z}_p)_+\to \St_{\Sigma_1}(\bb{Z}_p)$, are finite.
	\end{lemma}

\noindent Let $\Sigma'\subset\Sigma$ be  non-empty subsets of $S$.
For the rest of the section we write 
\[
\St_{\Sigma}^\mrm{ct}(\widehat{E}^\times_{\Sigma'}):=\St_{\Sigma'}^\mrm{ct}(\widehat{E}^\times)\otimes_{\Z_p}\St^\mrm{ct}_{\Sigma\setminus\Sigma'}(\bb{Z_\p}).
\] 
Moreover, we consider the image of $c_{A,\epsilon}^{\Sigma}$ under the composition
\[
\mrm{H}^q\big(X_G^{\Sigma},\St_\Sigma(\bb{Z}_p),\Z_p\big)\too  \mrm{H}^q\big(X_G^{\Sigma},\St^{\cont}_\Sigma(\widehat{E}^\times_{\p})\big)\too \mrm{H}^q\big(X_G^{\Sigma},\St^{\cont}_\Sigma(\widehat{E}^\times_{\p}),\widehat{A}(E_{\p})\big).
\]

\begin{lemma}\label{Step2}
	 Let $\Sigma\subseteq S$ be a non-empty subset and $\p\in\Sigma$.
	The image of $c_{A,\epsilon}^{\Sigma}$ under the boundary 
	\[
	\partial_\p\colon \mrm{H}^q\big(X_G^{\Sigma},\St_{\Sigma}^{\cont}(\widehat{E}^\times_\p),\widehat{A}(E_{\p})\big)\too \mrm{H}^{q+1}\big(X_G^{\Sigma},\St^{\cont}_{\Sigma\setminus\{\p\}}(\bb{Z}_p),\widehat{A}(E_{\p})\big),
	\]
	obtained from the short exact sequence defining $\mathfrak{E}_\p$, is torsion. 
\end{lemma}
\begin{proof} 
Consider the commutative diagram
\[\xymatrix{
\mrm{H}^q\big(X_G^\Sigma,\St^{\cont}_\Sigma(\Z_p),\Z_p\big)_\pi^\epsilon\ar[r]\ar[d] & \mrm{H}^q\big(X_G^\Sigma,\St^{\cont}_\Sigma(\widehat{E}^\times_\p),\widehat{A}(E_\p)\big)_\pi^\epsilon\ar[r]^-{\partial_\p}\ar[d] & \HH^{q+1}(X_G^\Sigma,\St^{\cont}_{\Sigma\setminus\{\p\}}(\Z_\p),\widehat{A}(E_{\p}))_\pi^\epsilon\ar[d]\\
\mrm{H}^q\big(X_G^{\{\p\}},\St^{\cont}_\p(\Z_p),\Z_p\big)_\pi^\epsilon\ar[r] & \mrm{H}^q\big(X_G^{\{\p\}},\St^{\cont}_\p(\widehat{E}^\times),\widehat{A}(E_\p)\big)_\pi^\epsilon\ar[r]^-{\partial_\p} & \mrm{H}^{q+1}\big(X_G^{\{\p\}},\Z_p,\widehat{A}(E_\p)\big)_\pi^\epsilon
}\]
where the vertical arrows are induced by $\Ev_{\Sigma\setminus\{\p\}}\colon  \bigotimes_{\p\in \Sigma\setminus\{\p\}} C_c(\scr{E}_\p,\bb{Z}_p)_+\to \St_{\Sigma\setminus\{\p\}}(\bb{Z}_p)$ and have finite kernel and cokernel by Lemma \ref{bijlemma}.
Moreover, the leftmost vertical arrow maps  $c^{\Sigma}_{A,\epsilon}$ to a non-zero multiple of $c^{\{\p\}}_{A,\epsilon}$.
Hence, the claim follows from Theorem \ref{Step1}.
\end{proof}

\noindent Let $\Sigma_1\subset\Sigma_2\subseteq\Sigma $ be subsets of $S$ with $\Sigma_2=\Sigma_1\cup\{\p\}$ for some $\p\not\in\Sigma_1$. Thanks to Lemma \ref{Step2} we can assume by  inductive hypothesis that we can lift $c_{A,\epsilon}^{\Sigma}$ to a class in 
\[
\widehat{c}_{A,\epsilon}^{\Sigma_1}\in \mrm{H}^{q}\big(X_G^\Sigma,\mathfrak{E}_{\Sigma_1}\otimes_{\Z_p}\St^{\cont}_{\Sigma\setminus\Sigma_1}(\Z_p),\widehat{A}(E_{\Sigma_1})\big)_{\pi}^{\epsilon}.
\]
 Then, we consider the image of $\widehat{c}_{A,\epsilon}^{\Sigma_1}$ under the composition
\[\xymatrix{
\mrm{H}^{q}\big(X_G^{\Sigma},\mathfrak{E}_{\Sigma_1}\otimes_{\Z_p}\St^{\cont}_{\Sigma\setminus\Sigma_1}(\Z_p),\widehat{A}(E_{\Sigma_1})\big)\ar[r]\ar@{.>}[rd]
&  \mrm{H}^{q}\big(X_G^{\Sigma},\mathfrak{E}_{\Sigma_1}\otimes_{\Z_p}\St^\mrm{ct}_{\Sigma\setminus\Sigma_1}(\widehat{E}^\times_\p),\widehat{A}(E_{\Sigma_1})\otimes_{\Z_p}\widehat{E}^\times_\p\big)\ar[d]\\
& \mrm{H}^{q}\big(X_G^{\Sigma},\mathfrak{E}_{\Sigma_1}\otimes_{\Z_p}\St^\mrm{ct}_{\Sigma\setminus\Sigma_1}(\widehat{E}^\times_\p),\widehat{A}(E_{\Sigma_2})\big).
}\]
We complete the inductive step in the next lemma.

\begin{lemma}\label{Step3}
	Let $\Sigma_1\subset\Sigma_2\subseteq\Sigma $ be subsets of $S$ such that $\Sigma_2=\Sigma_1\cup\{\p\}$ for some $\p\not\in\Sigma_1$, then there exists a non-zero multiple of $\widehat{c}_{A,\epsilon}^{\Sigma_1}$ that can be lifted with respect to 
	\[
	\mrm{H}^{q}\big(X_G^\Sigma,\mathfrak{E}_{\Sigma_2}\otimes_{\Z_p}\St^\mrm{ct}_{\Sigma\setminus\Sigma_2}(\Z_p),\widehat{A}(E_{\Sigma_2})\big)_{\pi}^{\epsilon}\too \mrm{H}^{q}\big(X_G^\Sigma,\mathfrak{E}_{\Sigma_1}\otimes_{\Z_p}\St^\mrm{ct}_{\Sigma\setminus\Sigma_1}(\widehat{E}^\times_\p),\widehat{A}(E_{\Sigma_2})\big)_{\pi}^{\epsilon}.
	\]
\end{lemma}
\begin{proof}
As in the proof of Lemma \ref{Step2}, we use Lemma \ref{bijlemma} to reduce to the case $\Sigma=\Sigma_2.$
We consider the commutative diagram
\[\xymatrix{
\HH^{q}(X_G^\Sigma,\mathfrak{E}_{\Sigma_1}\otimes_{\Z_p}\St^\mrm{ct}_{\p}(\Z_\p),\widehat{A}(E_{\Sigma_1}))_\pi^\epsilon\ar[r]\ar[d] &  \HH^{q}(X_G^\Sigma,\St^{\cont}_{\Sigma}(\widehat{E}^\times_{\Sigma_1}),\widehat{A}(E_{\Sigma_1}))_\pi^\epsilon\ar[d]\\
\HH^{q}(X_G^\Sigma,\mathfrak{E}_{\Sigma_1}\otimes_{\Z_p}\St^{\cont}_{\p}(\widehat{E}_{\p}^{\times}),\widehat{A}(E_{\Sigma}))_\pi^\epsilon\ar[r]\ar[d]^-{\partial_\p} & \HH^{q}(X_G^\Sigma,\St^{\cont}_{\Sigma}(\widehat{E}^\times),\widehat{A}(E_{\Sigma}))_\pi^\epsilon\ar[d]^-{\partial_\p}\\
\HH^{q+1}(X_G^\Sigma,\mathfrak{E}_{\Sigma_1},\widehat{A}(E_{\Sigma}))_\pi^\epsilon\ar[r] & \HH^{q+1}(X_G^\Sigma,\St^{\cont}_{\Sigma_1}(\widehat{E}^\times),\widehat{A}(E_{\Sigma}))_\pi^\epsilon
}\]
	where the horizontal maps are induced by the inclusion $\St^{\cont}_{\Sigma_1}(\widehat{E}^\times)\into \mathfrak{E}_{\Sigma_1}$.
	By inductive hypothesis, the top horizontal map sends $\widehat{c}_{A,\epsilon}^{\Sigma_1}$ to the image of $c_{A,\epsilon}^{\Sigma}$ in $\HH^{q}(X_G^\Sigma,\St^{\cont}_{\Sigma}(\widehat{E}^\times_{\Sigma_1}),\widehat{A}(E_{\Sigma_1}))_\pi^\epsilon$.
As the horizontal maps have finite kernel by Lemma \ref{inj}, the claim follows from Lemma \ref{Step2}. 
\end{proof}


\section{Anticyclotomic p-adic L-functions}
Following (\cite{Felix}, \cite{BG2}) we attach an anticyclotomic $p$-adic $L$-functions to the elliptic curve $A_{/F}$ and the quadratic extension $E/F$.
The main result of this section is a $p$-adic Gross--Zagier formula (Theorem \ref{GZtheorem}) relating higher derivatives of these $p$-adic $L$-functions to plectic $p$-adic invariants.

\subsection{Homogeneous spaces and homomorphisms}
Let $H$ be a topological group and $X$ a homogeneous $H$-space, i.e. a topological space $X$ with a transitive continuous $H$-action $H \times X \too X$
such that for every $x\in X$ the map
\[
H\too X,\quad h \mapsto h.x 
\]
induces a homeomorphism $H/H_x \too X$ where $H_x$ is the stabilizer of $x$.
For any abelian topological group $N$ endowed with the trivial $H$-action, the short exact sequence of $H$-modules
\[
0 \too N \too C(X,N) \too C(X,N)/N \too 0
\]
induces the boundary map
\[
\partial\colon \big(C(X,N)/N\big)^{H} \too \HH^{1}(H,N)=\Hom(H,N).
\]
Explicitly, the map $\partial$ is given as follows: given an element $\bar{f}\in (C(X,N)/N)^{H}$ we let $f\in C(X,N)$ be an arbitrary lift of it.
For every $h\in H$ and $x\in X$ the element $\varphi_{\bar{f}}(h):=f(h. x)-f(x)\in N$ is independent of $x$ and the chosen lift. Then we have
\[
\partial(\bar{f})(h)=\varphi_{\bar{f}}(h).
\]
\noindent
In particular, we deduce that $\partial$ takes values in the group $\Hom_{\cont}(H,N)$ of continuous homomorphism from $H$ to $N$. Moreover, $\partial(\bar{f})$ is trivial on the stabilizer $H_x$.

\begin{lemma}\label{homogeneous}
The map $\partial$ induces an isomorphism
\[
\partial\colon \big(C(X,N)/N\big)^H \xlongrightarrow{\sim} \big\{\varphi \in \Hom_{\cont}(H,N)\mid \varphi(h)=0\ \ \forall\ h\in H_x\big\}.
\]
\end{lemma}
\begin{proof}
	Since $H$ acts transitively on $X$, the natural inclusion $N^H\to C(X,N)^H$ is an isomorphism.
	Hence, the homomorphism $\partial$ is injective. Now, suppose $\varphi\colon H \to N$ is a continuous homomorphism that is trivial on $H_x$ for one (and thus for all) $x\in X$.
	Then $\varphi$ descends to a continuous function $f_\varphi \colon X \to N$ whose image in $C(X,N)/N$ is $H$-invariant and the lemma follows.
\end{proof}

\subsubsection{Torus action.}
The local embedding $\psi\colon E_\p \into B_\p\cong M_2(F_\p)$, provides a structure of a one-dimensional $E_\p$-vector space on $F_\p^2$. The torus $T_\p$ acts simply transitively on $\PP(F_\p)$ because $E_\p^\times$ acts simply transitively on $E_\p\setminus\{0\}$.
Moreover, $\PP(F_\p)$ is a principal homogeneous $T_\p$-space because both $\PP(F_\p)$ and $T_\p$ are compact.
We are interested in the $T_\p$-invariant function
\[
f_{\psi,\p}:=\Phi_\p([\tau_{\psi,\p}]-[\bar{\tau}_{\psi,\p}])\in C(\PP(F_\p),E_\p^{\times})/E_\p^{\times}.
\]

\begin{lemma}\label{keyidentity}
For all $t\in T_\p$ we have
\[
\partial(f_{\psi,\p})(t)=t^{1-\sigma_\p}.
\]
\end{lemma}
\begin{proof}
To compute $\partial(f_{\psi,\p})$ we are allowed to replace the principal homogeneous $T_\p$-space $\PP(F_\p)$ by an isomorphic one. The principal homogeneous $T_\p$-space structure on $\PP(F_\p)$ depends on the embedding $\psi$ and the fixed isomorphism $\iota\colon B_\p\cong M_2(F_\p)$, thus our strategy is to find another embedding $B_\p\hookrightarrow\mrm{M}_2(E_\p)$ giving a simple expression for the function $f_{\psi,\p}$ and the $T_\p$-action.

\noindent For convenience we identify $E_\p$ with its image under $\psi$, then we can choose an element $u\in B_\p$ such that $u^2=1$, $B_\p=E_\p\oplus E_\p u$, and $u$ anticommutes with the elements of $E_\p$, i.e. $u\cdot z=z^{\sigma_\p}\cdot u$ for all $z\in E_\p$.  Therefore, $B_\p$ is a rank two $E_\p$-module on which $B_\p$ acts by multiplication on the right. If we write elements of $B_\p$ as $a+bu$ with $a,b\in E_\p$, we obtain a local embedding 
\[
\iota'\colon  B_\p\hookrightarrow\mrm{M}_2(E_\p),\qquad a+bu\mapsto \begin{pmatrix}
	a&b\\
	b^{\sigma_\p}& a^{\sigma_\p}
\end{pmatrix}.
\]
By Skolem-Noether theorem, there exists $A\in\mrm{GL}_2(E_\p)$ such that $\iota'(\cdot)=A^{-1}\iota(\cdot)A$. The action of $A\in\mrm{GL}_2(E_\p)$ gives a $T_\p$-equivariant isomorphism $\alpha\colon \PP(E_\p)\overset{\sim}{\to} \PP(E_\p)$ if we endow the domain with $T_\p$-action given by $\iota'\circ\psi$ and the target with the $T_\p$-action given by $\iota\circ\psi$. The isomorphism maps  $0$ to $\tau_{\psi,\p}$ and $\infty$ to $\bar{\tau}_{\psi,\p}$. Moreover, if we set $X:=A^{-1}(\PP(F_\p))$, then $f_{\psi,\p}'=f_{\psi,\p}\circ\alpha$ satisfies
\[ 
f_{\psi,\p}'\in C(X,E_\p^\times)/E_\p^\times,\qquad f_{\psi,\p}'(x)=x.
\]
Finally, as any $t\in T_\p$ acts on $X$ through multiplication by 
$t^{1-\sigma_\p}$, the claim follows.
\end{proof}
\begin{remark}\label{imprem}
	The endomorphism of the torus $T_\p$ given by 
	\[
	t \mapsto\ t^{1-\sigma_\p} \pmod{ F_\p^{\times}}
	\]
	equals the squaring map $t\mapsto t^2$. 
\end{remark}

\subsection{Characters of the torus}
Let $E_{\mathfrak{c},\mbox{\tiny$S$}}$ denote the union of the narrow ring class fields of $E$ of conductor $\mathfrak{c}\cdot p_S^n$ for $n\geq 0.$
Then, the Artin map induces an isomorphism
\[
\rec_E\colon \overline{T(F)}^+\backslash T(\A^{\infty})/U(\mathfrak{c})^{S}\xlongrightarrow{\sim} \Gal\big(E_{\mathfrak{c},\mbox{\tiny$S$}}/E\big):=\cal{G}_{\mathfrak{c},\mbox{\tiny $S$}}
\]
where $\overline{T(F)}^+$ represents the closure of $T(F)^+$ in $T(\A^{\infty})/U(\mathfrak{c})^{S}$.

\noindent For $N$ a $T(F)^+$-module and $?\in \{\emptyset,c\}$, we let $T(F)^+$ act on the space of continuous functions $C_?\big(T(\A^{\infty})/U^{S}(\mathfrak{c}),N)$ via 
$(t\cdot f)(x)=t\cdot f(t^{-1}x)$.
Generalizing slightly previous notation,  we set
\begin{align*}
\mrm{H}_d(X_T(\mathfrak{c}),N)&:=\HH_d\big(T(F)^+,C_c\big(T(\A^{\infty})/U(\mathfrak{c}),N\big)\big)\\
\intertext{and}
\mrm{H}^{d}\big(X_T(\mathfrak{c}),N\big)&:=\mrm{H}^{d}\big(T(F)^+,C(T(\A^\infty)/U(\mathfrak{c}),N\big);
\end{align*}
i.e.~we allow non-trivial coefficients.

\noindent Let $R$ be a profinite $\Z_p$-algebra and $\chi\colon \cal{G}_{\mathfrak{c},\mbox{\tiny $S$}} \to R^{\times}$
a continuous group homomorphism.
By writing $\chi$ as a product $\chi=\prod_\q \chi_\q$ of local characters $\chi_\q\colon T_\q \to R^\times$ for all primes $\q$ of $F$, we may view $\chi$ as a $T(F)^+$-invariant element of 
\[
C\big(T(\A^{S,\infty})/U(\mathfrak{c})^S,\ \otimes_{\p\in S}C(T_{\p},R)\big),
\] where the tensor product is taken over $R$. 
For any topological abelian group $N$, we define $C_\plectic(T_S,N)$ to be the quotient of $C(T_S,N)$ by the subspace of all functions which are constant with respect to some variable, i.e.
\[
C_\plectic(T_S,N):=C(T_S,N)/ \sum_{\Sigma\subsetneq S} C(T_{\Sigma},N)
\]

\begin{definition}
Similarly to Definition \ref{twisted}, we define a homology class associated to $\chi$ by
	\[
	\Theta_{\chi}:=\chi\cap\vartheta\in \mrm{H}_{q}\big(X_T(\mathfrak{c}),\ \otimes_{\p\in S}C(T_{\p},R)\big),
	\]	
and denote by $\widetilde{\Theta}_{\chi}$ its image under the map in homology induced by $\otimes_{\p\in S}C(T_{\p},R)
	\to C_\plectic(T_{S},R).$
\end{definition}
\noindent Fix a closed ideal $I\subseteq R$ and for any $R$-module $N$ put $\overline{N}:=N/I N$. Under some assumptions, we aim at relating  $\Theta_{\chi}$ to the class $\theta_{\overline{\chi}}\in \mrm{H}_q(X_T(\mathfrak{c}),\overline{R})$ associated to $\overline{\chi}:=\chi\pmod{I}$ as in Section \ref{twistedfun}.  Suppose that for every $\p\in S$
\[
\chi_{\p} \equiv 1 \pmod{I}.
\]
Then, the functions $\mrm{d}\chi_\p\colon T_\p \to \overline{I}$, given by $t\mapsto \chi_\p(t)-1 \pmod{I^2}$,
are group homomorphisms and can be viewed as cohomology classes  as in Lemma \ref{homogeneous}
\begin{equation}
\mrm{d}\chi_\p\in \mrm{H}^{0}\big(T_\p,C(T_\p,\overline{I})/\overline{I}\big).
\end{equation}
By further assuming that  $I^{r+1}=0$, multiplication induces a $T_S$-equivariant and $R$-linear map
$\otimes_{\p\in S}C(T_\p,\overline{I})\to C(T_S,I^r)$, in turn giving a homomorphism
\begin{equation}
\mu\colon \bigotimes_{\p\in S}C(T_\p,\overline{I})/\overline{I}\too C_\plectic(T_S,I^{r}).
\end{equation}

\noindent Unravelling the definitions one easily deduces the following lemma.
\begin{lemma}\label{characters}
 In the homology group $\mrm{H}_{q}\big(X_T^S(\mathfrak{c}), C_\plectic(T_S,R)\big)$, the following equality holds
\[
\widetilde{\Theta}_{\chi}=  \mu_{\ast}\big(\mrm{d}\chi_{\p_1}\cup\ldots\cup\mrm{d}\chi_{\p_{r}}\big)\ \cap \ \vartheta_{\overline{\chi}}.
\]
In particular, $\widetilde{\Theta}_{\chi}=0$ if the $r$-th power $I^r$ equals zero.
\end{lemma}

\begin{remark}
For locally constant characters, Lemma \ref{characters} is a special case of (\cite{BG2}, Proposition 3.6).
One can easily generalize the full statement of (\cite{BG2}, Proposition 3.6) to the profinite setting using arguments as in (\cite{DS}, Section 3).
\end{remark}

\subsection{Construction of the p-adic L-function}
We keep the same notations as in the previous section and further define
\[
\scr{M}_\plectic(T_S):=\mrm{Hom}_{\bb{Z}_p}\big(\otimes_{\p\in S}C(T_\p,\bb{Z}_p)/\bb{Z}_p,\ \bb{Z}_p\big).
\]
With similar arguments as in the proof of Lemma \ref{restrictioniso}, one proves that
there exists an $R$-linear and $T_S$-equivariant integration pairing
	\begin{equation}\label{intpair}
	\int_{T_S}\colon\scr{M}_\plectic(T_S)\times C_\plectic(T_S,R) \too R.
	\end{equation}
As $\PP(F_S)$ is principal homogeneous $T_S$-space, the choice of a base-point $x\in\bb{P}^1(F_S)$ induces 
\[
\iota_x\colon \bb{P}^1(F_S)\xrightarrow{\sim} T_S,
\]
and any other such identification is obtained translating by an element of $T_S$. We obtain an isomorphism $\iota_x^\ast\colon\otimes_{\p\in S}C(T_\p,\Z_p)/\Z_p\xlongrightarrow{\sim}\St_S^{\cont}(\Z_p)$,
which together with restriction along $\psi$, produces
\[
\iota_x^{\ast}\colon \mrm{H}^{q}\big(X_G^{S};\St_{S}(\bb{Z}_p),\Z_p\big)
\too \mrm{H}^{q}\big(X_T(\mathfrak{c}),\scr{M}_\plectic(T_S)\big).
\]

\begin{definition}
	Consider  the completed group algebra $\Z_p\llbracket\cal{G}_{\mathfrak{c},\mbox{\tiny $S$}}\rrbracket$ and the universal character  $\chi_{\mathfrak{c}}\colon \cal{G}_{\mathfrak{c},\mbox{\tiny $S$}} \to \Z_p\llbracket\cal{G}_{\mathfrak{c},\mbox{\tiny $S$}}\rrbracket^\times$.
The anticyclotomic $p$-adic $L$-function attached to the tuple ($A_{/E},\epsilon,\mathfrak{c},S$) is  
\[
\scr{L}_S^{\epsilon}(A/E)_\mathfrak{c}:=\iota_{x}^{\ast}(c_{A,\epsilon}^S)\cap\widetilde{\Theta}_{\chi_{\mathfrak{c}}} \in \Z_p\llbracket\mathcal{G}_{\mathfrak{c},\mbox{\tiny $S$}}\rrbracket
\]
where the cap-product is induced by the integration pairing \eqref{intpair}.
\end{definition}
\begin{remark}\label{uniquenessL-function}
	The $p$-adic $L$-function $\scr{L}_S^{\epsilon}(A/E)_\mathfrak{c}$ is independent on the choice of base-point up to multiplication by a group-like element of $\Z_p\llbracket\cal{G}_{\mathfrak{c},\mbox{\tiny $S$}}\rrbracket$ in the image of $T_S$. Indeed, the identification $\iota_x:\bb{P}^1(F_S)\xrightarrow{\sim} T_S$ is unique up to translation by an element of $T_S$, and the homology class $\Theta_{\chi_{\mathfrak{c}}} $ actually takes values in $\mrm{Hom}_\mrm{gr}(T_S,R^\times)\subseteq C(T_S,R)$ for  $R=\Z_p\llbracket\cal{G}_{\mathfrak{c},\mbox{\tiny $S$}}\rrbracket$. 
\end{remark}

\subsection{Special value formulas}
We recall a special value formula for Rankin-Selberg $L$-functions due to Waldspurger. We consider an idele class character $\chi\colon\bb{A}_E^\times/E^\times\rightarrow\bb{C}^\times$
such that $\chi_{\lvert\bb{A}_F^\times}\equiv 1$, and let $\Pi_A=\otimes_v\Pi_v$ denote the cuspidal automorphic representation of $\mrm{PGL}_2(\bb{A}_F)$ associated to $A_{/F}$ by modularity. We write $L(s,\Pi_A,\chi)$ for the Rankin-Selberg $L$-function associated to $\Pi_A$ and $\chi$, normalized to have the center at $s=1/2$. Denoting by $\eta:\bb{A}_F^\times/F^\times\to\bb{C}^\times$ the quadratic character associated to $E/F$ by class field theory, the set  
\[
\Sigma(A,\chi)=\Big\{v\Big\lvert\ \epsilon\Big(\frac{1}{2},\Pi_v,\chi_v\Big)\not=\chi_v\cdot\eta_v(-1)\Big\}
\]
of places of $F$ is finite, and it is related to the global root number by the formula
\[
\epsilon\Big(\frac{1}{2},\Pi_A,\chi\Big)=(-1)^{\lvert\Sigma(A,\chi)\rvert}.
\]
Recall that the fixed quaternion algebra $B/F$ was chosen to have  ramification set equal to
\[
\Sigma(B):=\big\{\mathfrak{q}\mid\mathfrak{n}^{\mbox{\tiny $-$}}\big\}\cup\big\{\infty_{n+1},\dots,\infty_t\big\}.
\] 
and such that $\Pi_A$ admits a Jacquet-Langlands transfer $\pi$ to it. Further, $B$ admits an embedding $\psi\colon E\hookrightarrow B$ determining a non-split torus $T=E^\times/F^\times$ in the algebraic group $G=B^\times/F^\times$. Then, we can consider the period integral $\ell(\cdot,\chi)\colon \pi\rightarrow\bb{C}$ given by
\[
\ell(\xi,\chi):=\int_{[T(\bb{A}_F)]}\xi(t)\cdot\chi(t)\ \mrm{d}t
\]
where $\mrm{d}t$ denotes the Haar measure on $[T(\bb{A}_F)]:=T(\bb{A}_F)/T(F)$ of total volume $1$.

\begin{remark}\label{stupidvanishing}
	If $\Sigma(A,\chi)\not=\Sigma(B)$, then $\ell(\cdot,\chi)\colon \pi\rightarrow\bb{C}$ is the zero functional. Indeed, $\ell(\cdot,\chi)$ factors as a product of local linear functionals and there exists a place $v$ for which the space of linear functionals $\mrm{Hom}_{E_v^\times}\big(\pi_v\otimes\chi_v,\bb{C}\big)$ is zero by work of Tunnell \cite{Tunnel} and Saito \cite{Saito}. 
\end{remark}

\noindent Let $\langle\cdot,\cdot\rangle_\mrm{Pet}$ denote the Petersson inner product for automorphic forms on the quotient space $(B\otimes\bb{A}_F)^\times/\bb{A}_F^\times B^\times$ with respect to the chosen Haar measure $\mrm{d}t$, and choose non-trivial hermitian forms $\langle\cdot,\cdot\rangle_v$ on $\pi_v$ satisfying the product formula $\langle\cdot,\cdot\rangle_\mrm{Pet}=\prod_v\langle\cdot,\cdot\rangle_v$. Then, for any place $v$ of $F$ we can define the local linear functional $\alpha(\cdot,\chi_v):\pi_v\rightarrow\bb{C}$ by 
\[
\alpha(\xi_v,\chi_v)=\frac{L\big(1,\eta_v\big)\cdot L\Big(1,\Pi_v,\mrm{ad}\Big)}{\zeta_{F_v}(2)\cdot L\Big(\frac{1}{2},\Pi_v,\chi_v\Big)}\cdot\int_{T(F_v)}\big\langle \pi_v(t)\xi_v,\xi_v\big\rangle_v\cdot\chi_v(t)\ \mrm{d}t.
\]
Suppose $\Sigma(A,\chi)=\Sigma(B)$, then for a non-zero decomposable vector $\xi\in \pi$ we have $\alpha(\xi_v,\chi_v)=1$ for all but finitely many places $v$ of $F$, and Waldspurger formula holds
	\[
	\big\lvert \ell(\xi,\chi)\big\rvert^2=\frac{\zeta_F(2)\cdot L\Big(\frac{1}{2},\Pi_A,\chi\Big)}{8\cdot L\big(1,\eta\big)^2 \cdot L\Big(1,\Pi_A,\mrm{ad}\Big)}\cdot\prod_v\alpha(\xi_v,\chi_v).
	\]
In the form we presented here, this formula can be found in (\cite{YZZ}, Theorem 1.4). It justifies our choice of anticyclotomic $p$-adic $L$-function.

\begin{theorem}\label{corolp-adic}
Let $\chi\colon \cal{G}_{\mathfrak{c},\mbox{\tiny $S$}} \to\overline{\bb{Q}}^\times$ be a finite order character. On the one hand, if $\chi$ does not ramify at every prime in $S$ or $\chi_\infty\neq\epsilon$, then 
\[
\chi\big(\scr{L}_{S}^\epsilon(A/E)_\mathfrak{c}\big)=0.
\]
On the other hand, if $\chi$ has conductor $\frak{c}\cdot\prod_{i=1}^r\frak{p}_i^{n_i}$ with $n_i\ge1$ for every $i$, and $\chi_\infty=\epsilon$, then 
	\[
	\chi\Big(\scr{L}_{S}^\epsilon(A/E)_\mathfrak{c}\Big)\neq 0\qquad\iff\qquad L\Big(\frac{1}{2},\Pi_A,\chi\Big)\neq 0.
	\]
\end{theorem}
\begin{proof}
		This is a special case of (\cite{BG2}, Theorem 5.8). For the convenience of the reader, we sketch the idea of the proof. First, if $\chi_\infty\not=\epsilon$, then the $p$-adic $L$-function clearly vanishes. Thus, it is more interesting to consider a finite order character $\chi$ with $\chi_\infty=\epsilon$. In this case it is possible to compare $\chi\big(\scr{L}_{S}^\epsilon(A/E)_\mathfrak{c}\big)$  with a value of the period integral $\ell(\cdot,\chi)$. When $\chi$ is ramified at primes dividing $\mathfrak{c}$ -- which is coprime to $\mathfrak{f}_A$ -- and at primes in a subset $\Sigma\subseteq S$, one can compute that
	\[
	\Sigma(A,\chi)=\big(S\setminus\Sigma\big)\cup\big\{\mathfrak{q}\mid\mathfrak{n}^{\mbox{\tiny $-$}}\big\}\cup\big\{\infty_{n+1},\dots,\infty_t\big\}.
	\]
	Therefore, if  $\Sigma\not=S$, then $\Sigma(A,\chi)\not=\Sigma(B)$ and the period integral $\ell(\cdot,\chi)$ vanishes for local reasons (Remark $\ref{stupidvanishing}$). While when $\Sigma=S$, one can conclude by invoking Waldspurger's formula and the explicit computations of \cite{FileMartinPitale}.
\end{proof}

\subsection{The p-adic Gross--Zagier formula}\label{GZ}
The vanishing of the anticyclotomic $p$-adic $L$-function to order at least $r=\lvert S\rvert$ at any character of $\mathcal{G}_{\mathfrak{c}}$ is proved in (\cite{Felix} $\&$ \cite{BG2}, Theorem 5.5) and it is a direct consequence of Lemma \ref{characters}. Concretely, it means that
\[
\scr{L}_S^{\epsilon}(A/E)_\frak{c}\in I^{r}
\]
where $I=\ker\big(\Z_p\llbracket\mathcal{G}_{\mathfrak{c},\mbox{\tiny $S$}}^{\mbox{\tiny $+$}}\rrbracket\to \Z_p[\mathcal{G}_{\mathfrak{c}}^{\mbox{\tiny $+$}}]\big)$ is the relative augmentation ideal.
Inspired by the work of Bertolini and Darmon \cite{CDuniformization}, we prove that the values of the $r$-th derivative of $\scr{L}_S^{\epsilon}(A/E)_\mathfrak{c}$ compute plectic $p$-adic invariants.
\begin{definition}
	For any character $\chi\colon\G_{\mathfrak{c}}\to\overline{\bb{Q}}^\times$ with $\chi_\infty=\epsilon$ we set
	\[
	\chi\Big(\partial^{r}\scr{L}_{S}^\epsilon(A/E)_\mathfrak{c}\Big):= \scr{L}_{S}^\epsilon(A/E)_\mathfrak{c}\otimes 1\quad\text{in}\quad I^r/I^{r+1}\otimes_\chi\overline{\bb{Q}},
	\]
	where we consider $I^r/I^{r+1}$ as a $\Z[\G_{\mathfrak{c}}]$-module. 
\end{definition}

\begin{remark}\label{well-definedderivative}
	As every $\p\in S$ is inert in $E/F$, the image of the Artin map from $T_S$ to $\mathcal{G}_{\mathfrak{c}}$ is trivial. It follows that the element $\chi\big(\partial^{r}\scr{L}_{S}^\epsilon(A/E)_\mathfrak{c}\big)$ is independent of the choice of base-point $x\in\bb{P}^1(F_S)$ made in the definition of the $p$-adic $L$-function (Remark \ref{uniquenessL-function}).
\end{remark}
\noindent Since the image of the local Artin map from $T_\p$ to $\mathcal{G}_{\mathfrak{c}}$ is trivial, the function
\[
\dd\rec_\p\colon \widehat{E}_\p^{\times} \to I/I^2,\qquad t\mapsto \rec_\p(t)-1 \pmod{I^2}
\]
is a group homomorphism.
Taking tensor products over all $\p\in S$ yields the homomorphism
\[
\dd\rec_S \colon \widehat{E}_{S,\otimes}^{\times} \too I^{r}/I^{r+1}
\]
which extends to a $\Z_p[\G_{\mathfrak{c}}]$-linear homomorphism
\[
\dd\rec_S \colon \widehat{E}_{S,\otimes}^{\times}\otimes_{\Z_p}\Z_p[\G_{\mathfrak{c}}]\too I^{r}/I^{r+1}.
\]
Then, our $p$-adic Gross--Zagier formula takes the following form.
\begin{theorem}\label{GZtheorem}
For any character $\chi\colon\G_{\mathfrak{c}}\to\overline{\bb{Q}}^\times$ with $\chi_\infty=\epsilon$ the equality
\[
2^{r}\cdot \chi\Big(\partial^{r}\scr{L}_{S}^\epsilon(A/E)_\mathfrak{c}\Big) = \dd\rec_S(\mrm{Q}_{A}^{\chi})
\]
holds in $I^{r}/I^{r+1}\otimes_{\chi}\overline{\bb{Q}}$.
\end{theorem}
\begin{proof}
Let $\partial^{r}\scr{L}_{S}^{\epsilon}(A/E)_\frak{c}$ denote the image of $\scr{L}_{S}^{\epsilon}(A/E)_\frak{c}$ in $I^r/I^{r+1}$, and consider the reduction modulo $I^{r+1}$ of the universal character
\[
\chi_r\colon \mathcal{G}_{\mathfrak{c},\mbox{\tiny $S$}}^{\mbox{\tiny $+$}} \to \big(\Z_p\llbracket\mathcal{G}_{\mathfrak{c},\mbox{\tiny $S$}}\rrbracket/I^{r+1}\big)^\times.
\]
Directly from the definitions, we can write $\partial^{r}\scr{L}_{S}^{\epsilon}(A/E)_\frak{c}
=\iota_{x}^{\ast}(c_{A,\epsilon}^S)\cap\widetilde{\Theta}_{\chi_r}$. Moreover, Lemma \ref{characters} applies to our setting with $R= \Z_p\llbracket\mathcal{G}_{\mathfrak{c},\mbox{\tiny $S$}}\rrbracket/I^{r+1}$ giving
\[
\widetilde{\Theta}_{\chi_r}
= \mu_*\big(\mrm{d}\chi_{r,\p_1}\cup\dots\cup \mrm{d}\chi_{r,\p_r}\big)\ \cap\ \Theta_{\overline{\chi}_r}.
\]
By Lemma \ref{keyidentity} and Remark \ref{imprem}, the following equality holds in $\mrm{H}^q(X_T(\mathfrak{c}),I^{r}/I^{r+1})$
\[
\iota_{x}^{\ast}(c^S_{A,\epsilon})\cup \mu_*\big( 2\mrm{d}\chi_{r,\p_1}\cup\dots\cup 2\mrm{d}\chi_{r,\p_r}\big)
=(\mrm{d}\rec_S)_\ast\big((\Psi^{\diamond}_{S})^{\ast}( c^S_{A,\epsilon})\big).
\]
Then, the claim follows because any character $\chi:\G_{\mathfrak{c}}\to\overline{\bb{Q}}^\times$ factors through $\overline{\chi}_r$.
\end{proof}


\bibliography{Plectic}
\bibliographystyle{alpha}

\end{document}